\documentclass[12pt]{amsart}
\usepackage{amsfonts}
\usepackage{amsthm}
\usepackage{amsmath}
\usepackage{amssymb}
\usepackage{amscd}
\usepackage[latin2]{inputenc}
\usepackage{t1enc}
\usepackage[mathscr]{eucal}
\usepackage{indentfirst}
\usepackage{graphicx}
\usepackage{graphics}
\usepackage{pict2e}
\usepackage{epic}
\numberwithin{equation}{section}
\usepackage[margin=2.9cm]{geometry}
\usepackage{epstopdf} 

\usepackage{multirow} 
\usepackage{array}
\usepackage{tikz-cd}

\usepackage{extarrows}

\usepackage{float} 

\usepackage{soul}
\usepackage{xcolor}

\theoremstyle{definition}
\newtheorem{theorem}{Theorem}[section]
\newtheorem{lemma}[theorem]{Lemma}
\newtheorem{corollary}[theorem]{Corollary}
\newtheorem{proposition}[theorem]{Proposition}
\newtheorem{definition}[theorem]{Definition}
\newtheorem{example}[theorem]{Example}

\newtheorem{Rem}[theorem]{Remark}

\newcommand{\im}{\operatorname{im}}

\begin{document}

\title[Bijections and Lawrence polytope]{A framework unifying some bijections for graphs and its connection to Lawrence polytopes}

\author[Changxin Ding]{Changxin Ding}

\address{Georgia Institute of Technology \\ School of Mathematics \\Atlanta, GA 30332-0160} 

\email{cding66@gatech.edu}

\begin{abstract}
Let $G$ be a connected graph. The Jacobian group (also known as the Picard group or sandpile group) of $G$ is a finite abelian group whose cardinality equals the number of spanning trees of $G$. The Jacobian group admits a canonical simply transitive action on the set $\mathcal{R}(G)$ of cycle-cocycle reversal classes of orientations of $G$. Hence one can construct combinatorial bijections between spanning trees of $G$ and $\mathcal{R}(G)$ to build connections between spanning trees and the Jacobian group. The BBY bijections and the Bernardi bijections are two important examples. In this paper, we construct a new family of such bijections that includes both. Our bijections depend on a pair of atlases (different from the ones in manifold theory) that abstract and generalize certain common features of the two known bijections. The definitions of these atlases are derived from triangulations and dissections of the Lawrence polytopes associated to $G$. The acyclic cycle signatures and cocycle signatures used to define the BBY bijections correspond to regular triangulations. Our bijections can extend to subgraph-orientation correspondences. Most of our results hold for regular matroids. We present our work in the language of fourientations, which are a generalization of orientations.  
\end{abstract}

\maketitle


\section{Introduction}

This section should be considered as a paper without proofs. We provide all of the relevant definitions and main results. We also explain motivations and propose open questions. We will provide all the proofs in the later sections. 

\subsection{Overview}\label{introduction}
Given a connected graph $G$, we build a new family of bijections between the set $\mathcal{T}(G)$ of spanning trees of $G$ and the set $\mathcal{R}(G)$ of equivalence classes of orientations of $G$ up to cycle and cocycle reversals. The new family of bijections includes the \emph{BBY bijection} (also known as the geometric bijection) constructed by Backman, Baker, and Yuen \cite{BBY}, and the \emph{Bernardi bijection}\footnote{The Bernardi bijection in \cite{Bernardi} is a subgraph-orientation correspondence. In this paper, by the Bernardi bijection we always mean its restriction to spanning trees.} in \cite{Bernardi}. 

These bijections are closely related to the \emph{Jacobian group} (also known as the \emph{Picard group} or \emph{sandpile group}) $\text{Jac}(G)$ of $G$. The group $\text{Jac}(G)$ and the set $\mathcal{T}(G)$ of spanning trees are equinumerous. Recently, many efforts\footnote{In this overview, we focus on the work where the torsors are defined via bijections. We will discuss the \emph{rotor-routing torsor} defined and studied in \cite{Hol08,CCG} later in Section~\ref{open}.}  have been devoted to making $\mathcal{T}(G)$ a \emph{torsor} for $\text{Jac}(G)$, i.e., defining a simply transitive action of $\text{Jac}(G)$ on $\mathcal{T}(G)$. In \cite{BW}, Baker and Wang  interpreted the Bernardi bijection as a bijection between $\mathcal{T}(G)$ and \emph{break divisors}. Since the set of break divisors is a canonical torsor for $\text{Jac}(G)$, the Bernardi bijection induces the \emph{Bernardi torsor}. In \cite{Yuen}, Yuen defined the geometric bijection between $\mathcal{T}(G)$ and break divisors of $G$. Later, this work was generalized in \cite{BBY} where Backman, Baker, and Yuen defined the BBY bijection between $\mathcal{T}(G)$ and the cycle-cocycle reversal classes $\mathcal{R}(G)$. The set $\mathcal{R}(G)$ was introduced by Gioan \cite{G1} and is known to be a canonical torsor for $\text{Jac}(G)$ \cite{B, BBY}. Hence any bijection between $\mathcal{T}(G)$ and $\mathcal{R}(G)$ makes $\mathcal{T}(G)$ a torsor. 
From the point of view in \cite{BBY}, replacing break divisors with $\mathcal{R}(G)$ provides a more general setting. In particular, we may also view the Bernardi bijection as a bijection between $\mathcal{T}(G)$ and $\mathcal{R}(G)$ and define the Bernardi torsor. 

Our work puts all the above bijections in the same framework. It is surprising because the BBY bijection and the Bernardi bijection rely on totally different parameters. The main ingredients to define the BBY bijection 
are an \emph{acyclic cycle signature} $\sigma$ and an \emph{acyclic cocycle signature} $\sigma^*$ of $G$. The BBY bijection sends spanning trees to $(\sigma,\sigma^*)$-compatible orientations, which are representatives of $\mathcal{R}(G)$. The Bernardi bijection relies on a ribbon structure on the graph $G$ together with a vertex and an edge as initial data. Although for planar graphs, the Bernardi bijection becomes a special case of the BBY bijection, they are different in general \cite{Yuen, BBY}. The main ingredients to define our new bijections are a \emph{triangulating atlas} and a \emph{dissecting atlas} of $G$. These atlases (different from the ones in manifold theory) abstract and generalize certain common features of the two known bijections. They are derived from \emph{triangulations} and \emph{dissections} of the \emph{Lawrence polytopes} associated to graphs. The acyclic cycle signatures and cocycle signatures used to define the BBY bijections correspond to \emph{regular} triangulations.

Our bijections extend to subgraph-orientation correspondences. The construction is similar to the one that extends the BBY bijection in \cite{D2}. The extended bijections have nice specializations to forests and connected subgraphs. 

Our results are also closely related to and motivated by K\'alm\'an's work \cite{K}, K\'alm\'an and T\'othm\'er\'esz's work \cite{KT1}, and Postnikov's work \cite{P} on \emph{root polytopes} of \emph{hypergraphs}, where the hypergraphs specialize to graphs, and the Lawrence polytopes generalize the root polytopes in the case of graphs.

Most of our results hold for \emph{regular matroids} as in \cite{BBY}. Regular matroids are a well-behaved class of matroids which contains graphic matroids and co-graphic matroids. The paper will be written in the setting of regular matroids. 

We present our theory using the language of \emph{fourientations}, which are a generalization of orientations introduced by Backman and Hopkins \cite{BH}. 

Our paper is organized as follows. 
\begin{enumerate}
    \item[\ref{regular}] We review some basics of regular matroids. 
    \item[\ref{four}] We introduce fourientations.
    \item[\ref{new framework}] We use fourientations to build the framework: a pair of atlases and the induced map. We also recall the BBY bijection and the Bernardi bijection as examples.  
    \item[\ref{bijection and atlas}] We define triangulating atlases and dissecting atlases and present our bijections. 
    \item[\ref{sign}] We use our theory to study signatures. In particular, we generalize acyclic signatures to triangulating signatures. 
    \item[\ref{Lawrence-intro}] We build the connection between the geometry of the Lawrence polytopes and the combinatorics of the regular matroid. 
    \item[\ref{intro-motivation}] We explain the motivations by showing how our work is related to \cite{BBY, K, KT1, P}. 
    \item[\ref{open}] We propose some open questions. 
    \item[\ref{combinatorial results}] We prove the results in Section~\ref{bijection and atlas}.
    \item[\ref{signature}] We prove the results in Section~\ref{sign}.
    \item[\ref{Lawrence}] We prove the results in Section~\ref{Lawrence-intro}.
\end{enumerate}

\subsection{Notation and terminology: regular matroids}\label{regular}
In this section, we introduce the definition of regular matroids, signed circuits, signed cocircuits, orientations, etc; see also \cite{BBY} and \cite{D2}. We assume that the reader is familiar with the basic theory of matroids; some standard references include \cite{O}.

A matrix is called \emph{totally unimodular} if every square submatrix has determinant $0$, $1$, or $-1$. A matroid is called \emph{regular} if it can be represented by a totally unimodular matrix over $\mathbb{R}$. 

Let $\mathcal{M}$ be a regular matroid with ground set $E$. We call the elements in $E$ \emph{edges}. Without loss of generality, we may assume $\mathcal{M}$ is represented by an $r\times n$ totally unimodular matrix $M$, where $r=\text{rank}(M)$ and $n=|E|$. Here we require $r>0$ to avoid an empty matrix. For the case $r=0$, most of our results are trivially true. 

For any circuit $C$ of the regular matroid $\mathcal{M}$, there are exactly two $\{0, \pm 1\}$-vectors in $\ker_\mathbb{R}(M)$ with support $C$. We call them \emph{signed circuits} of $\mathcal{M}$, typically denoted by $\overrightarrow{C}$. Dually, for any cocircuit $C^*$, there are exactly two $\{0, \pm 1\}$-vectors in $\im_\mathbb{R}(M^T)$ with support $C^*$. We call them \emph{signed cocircuits} of $\mathcal{M}$, typically denoted by $\overrightarrow{C^*}$. The notions of signed circuit and signed cocircuit are intrinsic to $\mathcal{M}$, independent of the choice of $M$ up to \emph{reorientations}. By a reorientation, we mean multiplying some columns of $M$ by $-1$. For the proofs, see \cite{SW}. 
These signed circuits make $\mathcal{M}$ an \emph{oriented matroid} \cite[Chapter 1.2]{BVSWZ}, so regular matroids are in particular oriented matroids. 

It is well known that the \emph{dual matroid} $\mathcal{M}^*$ of a regular matroid $\mathcal{M}$ is also regular. There exists a totally unimodular matrix $M^*_{(n-r)\times n}$ such that the signed circuits and signed cocircuits of $\mathcal{M^*}$ are the signed cocircuits and signed circuits of $\mathcal{M}$, respectively. For the details, see \cite{SW}. The matrix $M^*$ should be viewed as a dual representation of $M$ in addition to a representation of $\mathcal{M^*}$. In particular, if we multiply some columns of $M$ by $-1$, then we should also multiply the corresponding columns of $M^*_{(n-r)\times n}$ by $-1$.

For any edge $e\in E$, we define an \emph{arc} $\overrightarrow{e}$ of $\mathcal{M}$ to be an $n$-tuple in the domain $\mathbb{R}^E$ of $M$, where the $e$-th entry is $1$ or $-1$ and the other entries are zero. We make the notion of arcs intrinsic to $\mathcal{M}$ in the following sense. If we multiply the $e$-th column of $M$ by $-1$, then an arc $\overrightarrow{e}$ will have the opposite sign with respect to the new matrix, but it is still the same arc of $\mathcal{M}$. So, the matrix $M$ provides us with a reference orientation for $E$ so that we know for the two opposite arcs associated with one edge which one is labeled by ``$1$''. The signed circuits $\overrightarrow{C}$ and signed cocircuits $\overrightarrow{C^*}$ can be viewed as sets of arcs in a natural way.  An \emph{orientation} of $\mathcal{M}$, typically denoted by $\overrightarrow{O}$, is a set of arcs where each edge appears exactly once. It makes sense to write $\overrightarrow{e}\in\overrightarrow{C}$, $\overrightarrow{C^*}\subseteq\overrightarrow{O}$, etc. In these cases we say the arc $\overrightarrow{e}$ is \emph{in} the signed circuit $\overrightarrow{C}$, the signed cocircuit $\overrightarrow{C^*}$ is \emph{in} the orientation $\overrightarrow{O}$, etc.

Now we recall the notion of \emph{circuit-cocircuit reversal (equivalence) classes} of orientations of $\mathcal{M}$ introduced by Gioan \cite{G1, G2}.
If $\overrightarrow{C}$ is a signed circuit in an orientation $\overrightarrow{O}$ of $\mathcal{M}$, then a \emph{circuit reversal} replaces $\overrightarrow{C}$ with the opposite signed circuit $-\overrightarrow{C}$ in $\overrightarrow{O}$. The equivalence relation generated by circuit reversals defines the \emph{circuit reversal classes} of orientations of $\mathcal{M}$. Similarly, we may define the \emph{cocircuit reversal classes}. The equivalence relation generated by circuit and cocircuit reversals defines the circuit-cocircuit reversal classes. We denote by $[\overrightarrow{O}]$ the circuit-cocircuit reversal class containing $\overrightarrow{O}$. It is proved in \cite{G2} that the number of circuit-cocircuit reversal classes of $\mathcal{M}$ equals the number of bases of $\mathcal{M}$.

Let $B$ be a basis of $\mathcal{M}$ and $e$ be an edge. If $e\notin B$, then we call the unique circuit in $B\cup \{e\}$ the \emph{fundamental circuit} of $e$ with respect to $B$, denoted by $C(B,e)$; if $e\in B$, then we call the unique cocircuit in $(E\backslash B)\cup \{e\}$ the \emph{fundamental cocircuit} of $e$ with respect to $B$, denoted by $C^*(B,e)$.

\emph{Graphic matroids} are important examples of regular matroids. Let $G$ be a connected finite graph with nonempty edge set $E$, where \emph{loops} and \emph{multiple edges} are allowed. By fixing a reference orientation of $G$, we get an \emph{oriented incidence matrix} of $G$. By deleting any row of the matrix, we get a totally unimodular matrix $M$ of full rank representing the graphic matroid $\mathcal{M}_G$ associated to $G$; see \cite{O} for details. Then edges, bases, signed circuits, signed cocircuits, arcs, orientations, and circuit-cocircuit reversal classes of $\mathcal{M}_G$ are edges, spanning trees, directed cycles, directed cocycles (bonds), arcs, orientations, and cycle-cocycle reversal classes of $G$, respectively. For direct definitions of these graph objects, see \cite[Section 2.1]{D2}.

\subsection{Notation and terminology: fourientations}\label{four}

It is convenient to introduce our theory in terms of \emph{fourientations}. Fourientations of graphs are systematically studied by Backman and Hopkins \cite{BH}. We will only make use of the basic notions but we define them for regular matroids. A fourientation $\overrightarrow{F}$ of the regular matroid $\mathcal{M}$ is a subset of the set of all the arcs. Symbolically, $\overrightarrow{F}\subseteq\{\pm\overrightarrow{e}: e\in E\}$. Intuitively, a fourientation is a choice for each edge of $\mathcal{M}$ whether to 
make it \emph{one-way oriented}, leave it \emph{unoriented}, or \emph{biorient} it. 
We denote by $-\overrightarrow{F}$ the fourientation obtained by reversing all the arcs in $\overrightarrow{F}$. In particular, the bioriented edges remain bioriented. We denote by $\overrightarrow{F}^c$ the set complement of $\overrightarrow{F}$, which is also a fourientation. Sometimes we use the notation $-\overrightarrow{F}^c$, which switches the unoriented edges and the bioriented edges in $\overrightarrow{F}$. See Figure~\ref{fourientation1} for examples of fourientations.

\begin{figure}[ht]
            \centering
            \includegraphics[scale=0.9]{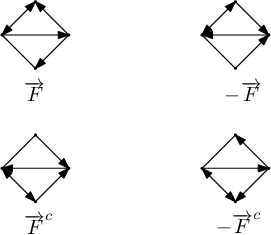}
            \caption{Examples of fourientations}
            \label{fourientation1}
\end{figure}

A \emph{potential circuit} of a fourientation $\overrightarrow{F}$ is a signed circuit $\overrightarrow{C}$ such that $\overrightarrow{C}\subseteq \overrightarrow{F}$. A \emph{potential cocircuit} of a fourientation $\overrightarrow{F}$ is a signed cocircuit $\overrightarrow{C^*}$ such that $\overrightarrow{C^*}\subseteq -\overrightarrow{F}^c$. 

\begin{figure}[ht]
            \centering
            \includegraphics[scale=1]{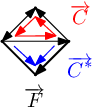}
            \caption{A potential circuit and a potential cocircuit of $\protect\overrightarrow{F}$}
            \label{fourientation2}
\end{figure}

\subsection{New framework: a pair of atlases $(\mathcal{A}, \mathcal{A^*})$ and the induced map $f_{\mathcal{A},\mathcal{A^*}}$}\label{new framework}

The BBY bijection studied in \cite{BBY} relies upon a pair consisting of an acyclic circuit signature and an acyclic cocircuit signature. We will generalize this work by building a new framework where the signatures are replaced by \emph{atlases} and the BBY bijection is replaced by a map $f_{\mathcal{A},\mathcal{A^*}}$. This subsection will introduce these new terminologies. 

\begin{definition}\label{def1}
Let $B$ be a basis of $\mathcal{M}$. \begin{enumerate}
    \item We call the edges in $B$ \emph{internal} and the edges not in $B$ \emph{external}.
    \item An \emph{externally oriented basis} $\overrightarrow{B}$ is a fourientation where all the internal edges are bioriented and all the external edges are one-way oriented. 
    \item An \emph{internally oriented basis} $\overrightarrow{B^*}$ is a fourientation where all the external edges are bioriented and all the internal edges are one-way oriented. 
    \item An \emph{external atlas} $\mathcal{A}$ of $\mathcal{M}$ is a collection of externally oriented bases $\overrightarrow{B}$ such that each basis of $\mathcal{M}$ appears exactly once. 
    \item An \emph{internal atlas} $\mathcal{A^*}$ of $\mathcal{M}$ is a collection of internally oriented bases $\overrightarrow{B^*}$ such that each basis of $\mathcal{M}$ appears exactly once.
\end{enumerate}
\end{definition}

See Figure~\ref{mapexample} for an example.

\begin{figure}[ht]
            \centering
            \includegraphics[scale=0.75]{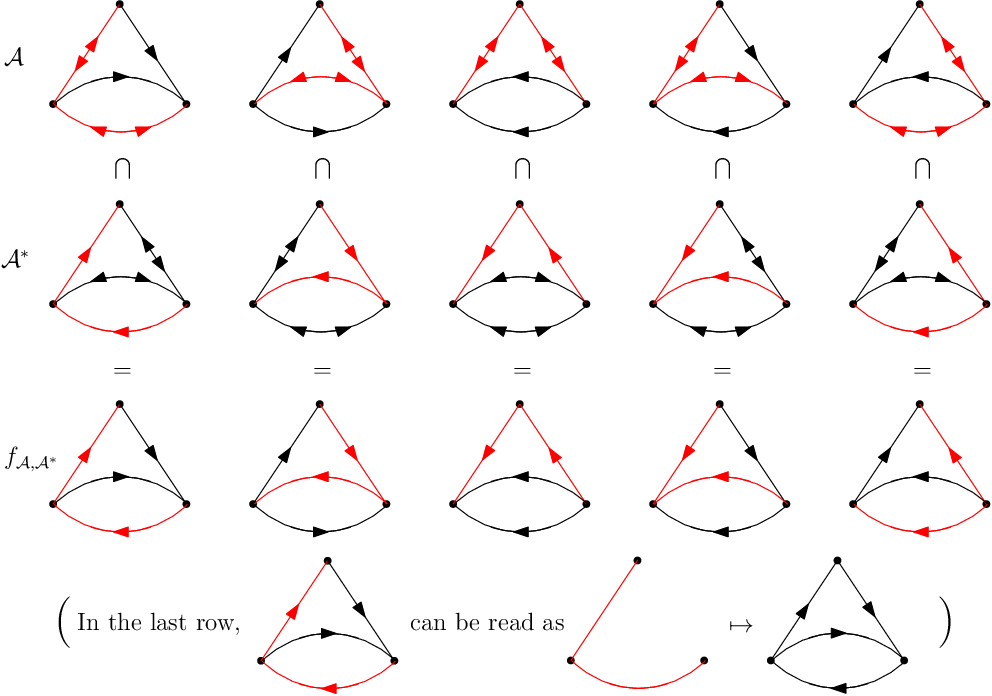}
            \caption{An example for Definition~\ref{def1} and \ref{def2}. The graph consists of 3 vertices and 4 edges. The bases (spanning trees) of the graph are in red. Each column consists of an externally oriented basis $\protect\overrightarrow{B}\in\mathcal{A}$, an internally oriented basis $\protect\overrightarrow{B^*}\in\mathcal{A}^*$, and their intersection $f_{\mathcal{A},\mathcal{A^*}}(B)=\protect\overrightarrow{B}\cap\protect\overrightarrow{B^*}$. }
            \label{mapexample}
\end{figure}  
 
Given an external atlas $\mathcal{A}$ (resp. internal atlas $\mathcal{A^*}$) and a basis $B$, by $\overrightarrow{B}$ (resp. $\overrightarrow{B^*}$) we always mean the oriented basis in the atlas although the notation does not refer to the atlas. 
\begin{definition}\label{def2}
For a pair of atlases $(\mathcal{A},\mathcal{A^*})$, we define the following map
    \begin{align*}
    f_{\mathcal{A},\mathcal{A^*}}:\{\text{bases of }\mathcal{M}\} & \to \{\text{orientations of }\mathcal{M}\} \\
    B & \mapsto \overrightarrow{B} \cap \overrightarrow{B^*} \text{ (where } \overrightarrow{B}\in\mathcal{A},\overrightarrow{B^*}\in\mathcal{A^*}\text{).}
    \end{align*}    
\end{definition}
We remark that, for any map $f$ from bases to orientations, there exists a unique pair of atlases $(\mathcal{A},\mathcal{A}^*)$ such that $f=f_{\mathcal{A},\mathcal{A^*}}$. This is because, given a basis $B$, there is no difference between \begin{itemize}
    \item orienting its external and internal edges separately, which $f_{\mathcal{A},\mathcal{A^*}}$ does, and  
    \item orienting all the edges of $\mathcal{M}$ simultaneously, which $f$ does.
\end{itemize}
So, the pair of atlases merely lets us view the map $f$ from a different perspective. However, from the main results of this paper, one will see why this new perspective interests us.

In the forthcoming Example~\ref{sigma atlas} and Example~\ref{B atlas}, we will rephrase the BBY bijection and the Bernardi bijection in our language. Before that, we recall the definitions of circuit signatures and cocircuit signatures introduced in \cite{BBY}.

\begin{definition}\label{def3}
Let $\mathcal{M}$ be a regular matroid. 
\begin{enumerate}
    \item A \emph{circuit signature}\footnote{For a graph $G$, circuit signatures are called cycle signatures, and cocircuit signatures are called cocycle signatures; also see the last paragraph of Section~\ref{regular}. } $\sigma$ of $\mathcal{M}$ is the choice of a direction for each circuit of $\mathcal{M}$. For each circuit $C$, we denote by $\sigma(C)$ the signed circuit we choose for $C$. By abuse of notation, we also view $\sigma$ as the set of the signed circuits chosen by $\sigma$.
    \item The circuit signature $\sigma$ is said to be \emph{acyclic} if whenever $a_C$ are nonnegative reals with $\sum_C a_C\sigma(C)=0$ in $\mathbb{R}^E$ we have $a_C=0$ for all $C$, where the sum is over all circuits of $\mathcal{M}$. 
    \item Cocircuit signatures $\sigma^*$ and acyclic cocircuit signatures are defined similarly. 
\end{enumerate}
\end{definition}

\begin{example}[Atlases $\mathcal{A_\sigma},\mathcal{A^*_{\sigma^*}}$ and the BBY map (bijection)]\label{sigma atlas} We rephrase the BBY map defined in \cite{BBY} as follows. Let $\sigma$ be a circuit signature of $\mathcal{M}$. We may construct an external atlas $\mathcal{A_\sigma}$ from $\sigma$ such that for each externally oriented basis $\overrightarrow{B}\in\mathcal{A_\sigma}$, each external arc $\overrightarrow{e}\in\overrightarrow{B}$ is oriented according to the signed fundamental circuit $\sigma(C(B,e))$. Similarly, we may construct an internal atlas $\mathcal{A}^*_{\sigma^*}$ from any cocircuit signature $\sigma^*$ such that for each internally oriented basis $\overrightarrow{B^*}\in\mathcal{A}^*_{\sigma^*}$, each internal arc $\overrightarrow{e}\in\overrightarrow{B^*}$ is oriented according to the signed fundamental cocircuit $\sigma^*(C^*(B,e))$. Then when the two signatures are acyclic, the map $f_{\mathcal{A_\sigma},\mathcal{A^*_{\sigma^*}}}$ is exactly the BBY map $\hat{\beta}$ in \cite[Theorem 1.3.1]{BBY}. 
\end{example}

\begin{example}[Atlases $\mathcal{A_\text{B}}, \mathcal{A}_q^*$ and the Bernardi map (bijection)]\label{B atlas}
The Bernardi map is defined for a connected graph $G$ equipped with a \emph{ribbon structure} (a.k.a \emph{combinatorial embedding}) and with initial data $(q,e)$, where $q$ is a vertex and $e$ is an edge incident to the vertex; see \cite[Section 3.2]{Bernardi} for details. Here we use an example (Figure~\ref{mapexample2}) to recall the construction of the bijection in the atlas language. The Bernardi bijection is a map from spanning trees to certain orientations. The construction makes use of the \emph{Bernardi tour} which starts with $(q,e)$ and goes around a given tree $B$ according to the ribbon structure. We may construct an external atlas $\mathcal{A_\text{B}}$ of $\mathcal{M}_G$ as follows. Observe that the Bernardi tour cuts each external edge twice. We orient each external edge toward the first-cut endpoint,  biorient all the internal edges of $B$, and hence get an externally oriented basis $\overrightarrow{B}$. All such externally oriented bases form the atlas $\mathcal{A_\text{B}}$. 

The internal atlas $\mathcal{A}_q^*$ of $\mathcal{M}_G$ is constructed as follows. For any tree $B$, we orient each internal edge away from $q$, biorient external edges, and hence get $\overrightarrow{B^*}\in\mathcal{A}_q^*$. We remark that $\mathcal{A}_q^*$ is a special case of $\mathcal{A^*_{\sigma^*}}$, where $\sigma^*$ is an acyclic cocycle signature \cite[Example 5.1.5]{Yuen2}. 

The map $f_{\mathcal{A}_\text{B},\mathcal{A}_q^*}$ is exactly the Bernardi map $\Phi$ in \cite[Section 3.2]{Bernardi}. 

\begin{figure}[ht]
            \centering
            \includegraphics[scale=0.9]{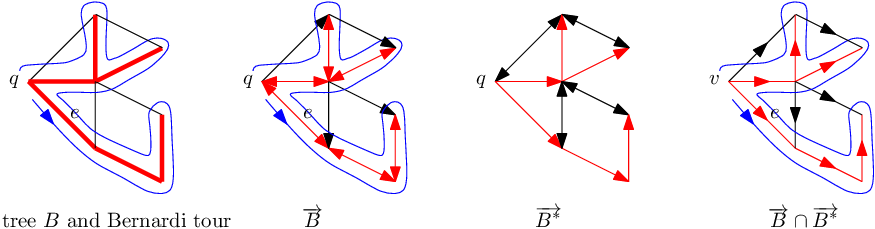}
            \caption{An example for the Bernardi map. The tree $B$ is in red.}
            \label{mapexample2}
\end{figure}
\end{example}

\subsection{Bijections and the two atlases}\label{bijection and atlas}

We will see in this subsection that the map $f_{\mathcal{A},\mathcal{A^*}}$ induces a bijection between bases of $\mathcal{M}$ and circuit-cocircuit reversal classes of $\mathcal{M}$ when the two atlases satisfy certain conditions which we call \emph{dissecting} and \emph{triangulating}. Furthermore, we will extend the bijection as in \cite{D2}. 

The following definitions play a central role in our paper. Although the definitions are combinatorial, they were derived from Lawrence polytopes; see Section~\ref{Lawrence-intro}.

\begin{definition}\label{key def}
Let $\mathcal{A}$ be an external atlas and $\mathcal{A^*}$ be an internal atlas of $\mathcal{M}$. 

\begin{enumerate}
    \item We call $\mathcal{A}$ \emph{dissecting} if for any two distinct bases $B_1$ and $B_2$, the fourientation $\overrightarrow{B_1}\cap(-\overrightarrow{B_2})$ has a potential cocircuit.
    \item We call $\mathcal{A}$ \emph{triangulating} if for any two distinct bases $B_1$ and $B_2$, the fourientation $\overrightarrow{B_1}\cap(-\overrightarrow{B_2})$ has no potential circuit.
    \item We call $\mathcal{A^*}$ \emph{dissecting} if for any two distinct bases $B_1$ and $B_2$, the fourientation $(\overrightarrow{B_1^*}\cap(-\overrightarrow{B_2^*}))^c$ has a potential circuit.
    \item We call $\mathcal{A^*}$ \emph{triangulating} if for any two distinct bases $B_1$ and $B_2$, the fourientation $(\overrightarrow{B_1^*}\cap(-\overrightarrow{B_2^*}))^c$ has no potential cocircuit.
\end{enumerate}
\end{definition}

\begin{Rem}
In Definition~\ref{key def}, being triangulating implies being dissecting. Indeed, the fourientation $\overrightarrow{B_1}\cap(-\overrightarrow{B_2})$ must contain a one-way oriented edge because of $B_1\neq B_2$. Then by Lemma~\ref{3-painting}, the absence of potential circuits implies the existence of a potential cocircuit. A similar argument works for the other fourientation.  
\end{Rem}

Now we are ready to present the first main result in this paper. 
\begin{theorem}\label{main1}
    Given a pair of dissecting atlases $(
\mathcal{A},\mathcal{A^*})$ of a regular matroid $\mathcal{M}$, if at least one of the atlases is triangulating, then the map 
    \begin{align*}
    \overline{f}_{\mathcal{A},\mathcal{A^*}}:\{\text{bases of }\mathcal{M}\} & \to \{\text{circuit-cocircuit reversal classes of }\mathcal{M}\} \\
    B & \mapsto [\overrightarrow{B} \cap \overrightarrow{B^*}]
    \end{align*}
is bijective, where $[\overrightarrow{B} \cap \overrightarrow{B^*}]$ denotes the circuit-cocircuit reversal class containing the orientation $\overrightarrow{B} \cap \overrightarrow{B^*}$.
\end{theorem}

\begin{example}[Example~\ref{sigma atlas} continued]\label{mainex}
One of the main results in \cite{BBY} is that the BBY map induces a bijection between bases and circuit-cocircuit reversal classes. We will see that both $\mathcal{A_\sigma}$ and $\mathcal{A^*_{\sigma^*}}$ are triangulating (Lemma~\ref{acyclic-tri}). Thus Theorem~\ref{main1} recovers this result.   
\end{example}

\begin{example}[Example~\ref{B atlas} continued]
Theorem~\ref{main1} also recovers the bijectivity of the Bernardi map for trees in \cite{Bernardi}. In \cite{Bernardi}, it is proved that the Bernardi map is a bijection between spanning trees and the \emph{$q$-connected outdegree sequences}. Baker and Wang \cite{BW} observed that the $q$-connected outdegree sequences are essentially the same as the \emph{break divisors}. Later in \cite{BBY}, the break divisors are equivalently replaced by cycle-cocycle reversal classes. We will see that the external atlas $\mathcal{A}_\text{B}$ is dissecting (Lemma~\ref{bernardidissect}). The internal atlas $\mathcal{A}_q^*$ is triangulating because it equals $\mathcal{A^*_{\sigma^*}}$ for some acyclic signature $\sigma^*$. Hence the theorem applies.

\end{example}

\begin{example}
In Theorem~\ref{main1}, if we only assume that the pair of atlases is dissecting, then the map $\overline{f}_{\mathcal{A},\mathcal{A^*}}$ is not necessarily bijective. For example, in Figure~\ref{mapexample}, one can check that the two atlases are dissecting but not triangulating. The two leftmost orientations (in the third row) are in the same circuit-cocircuit reversal class. 
\end{example}

In \cite{D2}, the BBY bijection is extended to a bijection $\varphi$ between subsets of $E$ and orientations of $\mathcal{M}$ in a canonical way. We also generalize this work by extending $f^{-1}_{\mathcal{A},\mathcal{A^*}}$ to $\varphi_{\mathcal{A},\mathcal{A^*}}$. The definition of the extension is the same as the one in \cite[Theorem 1.2(1)]{D2}, except that the map $f_{\mathcal{A},\mathcal{A^*}}$ is more general in this paper. 

\begin{definition}[The map $\varphi_{\mathcal{A},\mathcal{A^*}}$]\label{def-ext}
We will define a map $\varphi_{\mathcal{A},\mathcal{A^*}}$ from orientations to subgraphs such that $\varphi_{\mathcal{A},\mathcal{A^*}}\circ f_{\mathcal{A},\mathcal{A^*}}$ is the identity map, and hence $\varphi_{\mathcal{A},\mathcal{A^*}}$ extends $f^{-1}_{\mathcal{A},\mathcal{A^*}}$.  We start with an orientation $\overrightarrow{O}$. By Theorem~\ref{main1}, we get a basis $B=\overline{f}_{\mathcal{A},\mathcal{A^*}}^{-1}([\overrightarrow{O}])$. Since $\overrightarrow{O}$ and $f_{\mathcal{A},\mathcal{A^*}}(B)$ are in the same circuit-cocircuit reversal class, one can obtain one of them by reversing disjoint\footnote{By disjoint signed circuits $\{\overrightarrow{C_i}\}_{i\in I}$ and cocircuits $\{\overrightarrow{C_j^*}\}_{j\in J}$, we mean that any two elements from the set $\{C_i:i\in I\}\cup\{C_j^*:j\in J\}$ are disjoint.} signed circuits $\{\overrightarrow{C_i}\}_{i\in I}$ and cocircuits $\{\overrightarrow{C_j^*}\}_{j\in J}$ in the other (see Lemma~\ref{orientation1}). Define \[\varphi_{\mathcal{A},\mathcal{A^*}}(\overrightarrow{O})=(B\cup \biguplus\limits_{i\in I}C_i)\backslash \biguplus\limits_{j\in J}C_j^*,\]where the symbol $\biguplus$ means disjoint union. See Figure~\ref{fig: phi map} for the relations among the objects in the definition. 
\end{definition}

\begin{figure}[H]
    \centering
\begin{tikzcd}
B \arrow[r, "f_{\mathcal{A},\mathcal{A^*}}", maps to] \arrow[d, "\text{add }C_i\text{ and remove }C_j^*"', dashed] & \overrightarrow{O'} \arrow[d, "\text{reverse }\overrightarrow{C_i}\text{ and } \overrightarrow{C_j^*}", dashed]  \\
S                                                   & \overrightarrow{O} \arrow[l, "\varphi_{\mathcal{A},\mathcal{A^*}}", maps to]
\end{tikzcd}
    \caption{A description of Definition~\ref{def-ext}. The orientation $\overrightarrow{O'}$ is the unique orientation in the circuit-cocircuit reversal class $[\overrightarrow{O}]$ and the image of $f_{\mathcal{A},\mathcal{A^*}}$. The subgraph $S$ is $\varphi_{\mathcal{A},\mathcal{A^*}}(\overrightarrow{O})$. The direction of the arrow connecting $\overrightarrow{O}$ and $\overrightarrow{O'}$ is immaterial.}
    \label{fig: phi map}
\end{figure}

The amazing fact here is that $\varphi_{\mathcal{A},\mathcal{A^*}}$ is a bijection, and it has two nice specializations besides $f^{-1}_{\mathcal{A},\mathcal{A^*}}$. 
\begin{theorem}\label{main2}
Fix a pair of dissecting atlases $(
\mathcal{A},\mathcal{A^*})$ of $\mathcal{M}$ with ground set $E$. Suppose at least one of the atlases is triangulating. 

(1) The map 
    \begin{align*}
    \varphi_{\mathcal{A},\mathcal{A^*}}:\{\text{orientations of }\mathcal{M}\} & \to \{\text{subsets of }E\} \\
    \overrightarrow{O} & \mapsto (B\cup \biguplus_{i\in I}C_i)\backslash \biguplus_{j\in J}C_j^*
    \end{align*}
is a bijection, where $B$ is the unique basis such that $f_{\mathcal{A},\mathcal{A^*}}(B)\in [\overrightarrow{O}]$, and the orientations $f_{\mathcal{A},\mathcal{A^*}}(B)$ and $\overrightarrow{O}$ differ by  disjoint signed circuits $\{\overrightarrow{C_i}\}_{i\in I}$ and cocircuits $\{\overrightarrow{C_j^*}\}_{j\in J}$.

(2) The image of the independent sets of $\mathcal{M}$ under the bijection $\varphi^{-1}_{\mathcal{A},\mathcal{A^*}}$ is a representative set of the circuit reversal classes of $\mathcal{M}$. 

(3) The image of the spanning sets of $\mathcal{M}$ under the bijection $\varphi^{-1}_{\mathcal{A},\mathcal{A^*}}$ is a representative set of the cocircuit reversal classes of $\mathcal{M}$. 
\end{theorem}

\begin{Rem}
We can apply Theorem~\ref{main2} to extend and generalize the Bernardi bijection; see Corollary~\ref{Bernardi-extend} for a formal statement. In \cite{Bernardi}, the Bernardi bijection is also extended to a subgraph-orientation correspondence. However, Bernardi's extension is different from the bijection $\varphi_{\mathcal{A}_\text{B},\mathcal{A}_q^*}$ in Theorem~\ref{main2} in general. 
\end{Rem}

\subsection{Signatures and the two atlases}\label{sign}
In Section~\ref{bijection and atlas}, we have seen that acyclic signatures $\sigma$ and $\sigma^*$ induce triangulating atlases $\mathcal{A}_\sigma$ and $\mathcal{A^*_{\sigma^*}}$, respectively, and hence we may apply our main theorems to the BBY bijection. In this section, we will define a new class of signatures, called \emph{triangulating signatures}, which are in one-to-one correspondence with triangulating atlases and generalize acyclic signatures. Note that in \cite{BBY}, the BBY map is proved to be bijective onto \emph{$(\sigma,\sigma^*)$-compatible orientations}, which are representatives of the circuit-cocircuit reversal classes. We will also generalize this result. In particular, we will reformulate Theorem~\ref{main1} and Theorem~\ref{main2} in terms of the signatures and the compatible orientations (for triangulating atlases).

Recall in Example~\ref{sigma atlas} that from signatures $\sigma$ and $\sigma^*$, we may construct atlases $\mathcal{A_\sigma}$ and $\mathcal{A^*_{\sigma^*}}$. It is natural to ask: (1) Which signatures induce triangulating atlases? (2) Is any triangulating atlas induced by a signature? 

The following definition and theorem answer these two questions. 

\begin{definition}
    \begin{enumerate}
        \item A circuit signature $\sigma$ is said to be \emph{triangulating} if for any $\overrightarrow{B}\in\mathcal{A_\sigma}$ and any signed circuit $\overrightarrow{C}\subseteq\overrightarrow{B}$, $\overrightarrow{C}$ is in the signature $\sigma$.
        \item A cocircuit signature $\sigma^*$ is said to be \emph{triangulating} if for any $\overrightarrow{B^*}\in\mathcal{A^*_{\sigma^*}}$ and any signed cocircuit $\overrightarrow{C^*}\subseteq\overrightarrow{B^*}$, $\overrightarrow{C^*}$ is in the signature $\sigma^*$.
    \end{enumerate}
\end{definition}

\begin{Rem}\label{atlas-free}
In an atlas-free manner, the definition of triangulating circuit signatures is as follows: a circuit signature $\sigma$ is said to be triangulating if for any basis $B$, any signed circuit that is the sum of signed fundamental circuits (for $B$) in $\sigma$ is also in $\sigma$ (see Lemma~\ref{fundamental}). A similar definition works for the cocircuit signatures.  
\end{Rem}

\begin{theorem}\label{trig-intro}
The maps
    \begin{align*}
    \alpha:\{\text{triangulating circuit sig. of }\mathcal{M}\} & \to \{\text{triangulating external atlases of }\mathcal{M}\} \\
    \sigma & \mapsto \mathcal{A_\sigma}
    \end{align*}
and
    \begin{align*}
    \alpha^*:\{\text{triangulating cocircuit sig. of }\mathcal{M}\} & \to \{\text{triangulating internal atlases of }\mathcal{M}\} \\
    \sigma^* & \mapsto \mathcal{A}^*_{\sigma^*}
    \end{align*}
are bijections. 
\end{theorem}

\begin{Rem}\label{rem-nonexample}
For a dissecting external atlas $\mathcal{A}$, it is possible for there to be no circuit signature $\sigma$ such that $\mathcal{A}_\sigma=\mathcal{A}$. Consider the example in Figure~\ref{mapexample}. The two parallel edges form a circuit $C$. If we have $\mathcal{A}_\sigma=\mathcal{A}$ for some circuit signature $\sigma$, then $\sigma(C)$ is both clockwise and counterclockwise due to the leftmost externally oriented basis and the rightmost externally oriented basis, respectively. 
\end{Rem}

\begin{Rem}
Acyclic signatures are all triangulating; see Lemma~\ref{acyclic-tri}. There exists a triangulating signature that is not acyclic; see Proposition~\ref{nonex}. 
\end{Rem}

A nice thing about the acyclic signatures is that the associated compatible orientations (defined below) form representatives of orientation classes (proved in \cite{BBY}). The triangulating signatures also have this property; see the proposition below.

\begin{definition}\label{def4}
Let $\mathcal{M}$ be a regular matroid, $\sigma$ be a circuit signature, $\sigma^*$ be a cocircuit signature, and $\overrightarrow{O}$ be an orientation of $\mathcal{M}$ 
\begin{enumerate}
    \item The orientation $\overrightarrow{O}$ is said to be \emph{$\sigma$-compatible} if any signed circuit in the orientation is in $\sigma$. 
    \item The orientation $\overrightarrow{O}$ is said to be \emph{$\sigma^*$-compatible} if any signed cocircuit in the orientation is in $\sigma^*$. 
    \item The orientation $\overrightarrow{O}$ is said to be \emph{($\sigma,\sigma^*$)-compatible} if it is both $\sigma$-compatible and $\sigma^*$-compatible.
\end{enumerate}
\end{definition}

\begin{proposition}\label{prop: triangulating sign rep}
Suppose $\sigma$ and $\sigma^*$ are \emph{triangulating} signatures. 
\begin{enumerate}
    \item The set of $(\sigma, \sigma^*)$-compatible orientations is a representative set of the circuit-cocircuit reversal classes of $\mathcal{M}$.
    \item The set of $\sigma$-compatible orientations is a representative set of the circuit reversal classes of $\mathcal{M}$.
    \item The set of $\sigma^*$-compatible orientations is a representative set of the cocircuit reversal classes of $\mathcal{M}$.
\end{enumerate}
\end{proposition}

To reformulate Theorem~\ref{main1} and Theorem~\ref{main2} in terms of signatures and compatible orientations, we write
\[\text{BBY}_{\sigma,\sigma^*}=f_{\mathcal{A_\sigma},\mathcal{A^*_{\sigma^*}}}\text{ and }\varphi_{\sigma,\sigma^*}=\varphi_{\mathcal{A_\sigma},\mathcal{A^*_{\sigma^*}}}.\]
They are exactly the BBY bijection in \cite{BBY} and the extended BBY bijection in \cite{D2} when the two signatures are acyclic. By the two theorems and a bit of extra work, we have the following theorems, which generalize the work in \cite{BBY} and \cite{D2}, respectively. 
\begin{theorem}\label{main1-sign}
Suppose $\sigma$ and $\sigma^*$ are \emph{triangulating} signatures of a regular matroid $\mathcal{M}$. The map $\text{BBY}_{\sigma,\sigma^*}$ is a bijection between the bases of $\mathcal{M}$ and the $(\sigma,\sigma^*)$-compatible orientations of $\mathcal{M}$. 
\end{theorem}
\begin{theorem}\label{main2-sign}
Suppose $\sigma$ and $\sigma^*$ are \emph{triangulating} signatures of a regular matroid $\mathcal{M}$ with ground set $E$.  

(1) The map 
\begin{align*}
\varphi_{\sigma,\sigma^*}:\{\text{orientations of }\mathcal{M}\} & \to \{\text{subsets of } E\} \\
\overrightarrow{O} & \mapsto (\text{BBY}_{\sigma,\sigma^*}^{-1}(\overrightarrow{O^{cp}})\cup \biguplus_{i\in I}C_i)\backslash \biguplus_{j\in J}C_j^*
\end{align*}
is a bijection, where $\overrightarrow{O^{cp}}$ the (unique) ($\sigma,\sigma^*$)-compatible orientation obtained by reversing disjoint signed circuits $\{\overrightarrow{C_i}\}_{i\in I}$ and signed cocircuits $\{\overrightarrow{C_j^*}\}_{j\in J}$ in $\overrightarrow{O}$.

(2) The map $\varphi_{\sigma,\sigma^*}$ specializes to the bijection
\begin{align*}
\varphi_{\sigma,\sigma^*}: \{\sigma\text{-compatible orientations}\} & \to \{\text{independent sets}\} \\
\overrightarrow{O} & \mapsto \text{BBY}_{\sigma,\sigma^*}^{-1}(\overrightarrow{O^{cp}})\backslash \biguplus_{j\in J}C_j^*.
\end{align*}

(3) The map $\varphi_{\sigma,\sigma^*}$ specializes to the bijection
\begin{align*}
\varphi_{\sigma,\sigma^*}:\{\sigma^*\text{-compatible orientations}\} & \to \{\text{spanning sets}\} \\
\overrightarrow{O} & \mapsto \text{BBY}_{\sigma,\sigma^*}^{-1}(\overrightarrow{O^{cp}})\cup \biguplus_{i\in I}C_i.
\end{align*}
\end{theorem}

The definition of triangulating signatures is somewhat indirect. However, in the case of \emph{graphs}, we have the following nice description of the triangulating \emph{cycle} signatures. We will show a proof given by Gleb Nenashev. We do not know whether a similar statement holds for regular matroids. 

\begin{theorem}\label{triangular}\cite{Nenashev}
A \emph{cycle} signature $\sigma$ of a \emph{graph} $G$ is triangulating if and only if for any three directed cycles in $\sigma$, their sum (as vectors in $\mathbb{Z}^E$) is not zero. 
\end{theorem}

\subsection{Lawrence polytopes and the two atlases}\label{Lawrence-intro}

In this subsection, we will introduce a pair of \emph{Lawrence polytopes} $\mathcal{P}$ and $\mathcal{P^*}$ associated to a regular matroid $\mathcal{M}$. We will see that dissections and triangulations of the Lawrence polytopes correspond to the dissecting atlases and triangulating atlases, respectively, which is actually how we derived Definition~\ref{key def}. We will also see that regular triangulations correspond to acyclic signatures. 

Readers can find some information on Lawrence polytopes in the paper \cite{BS} and the books \cite{BVSWZ,Z}. The Lawrence polytopes defined for regular matroids in this paper were rediscovered by the author in attempts to define a dual object to the root polytope studied in \cite{P}; see Section~\ref{intro-motivation} for details.

Recall that $M_{r\times n}$ is a totally unimodular matrix representing $\mathcal{M}$. 

\begin{definition}
\begin{enumerate}
    \item We call \[\begin{pmatrix} M_{r\times n} & {\bf 0} \\  I_{n\times n} &  I_{n\times n} \end{pmatrix}\] the \emph{Lawrence matrix}, where $I_{n\times n}$ is the identity matrix. The columns of the Lawrence matrix are denoted by $P_1, \cdots, P_n$, $P_{-1}, \cdots, P_{-n}\in\mathbb{R}^{n+r}$ in order. 
    
    \item The \emph{Lawrence polytope} $\mathcal{P}\subseteq \mathbb{R}^{n+r}$ of $\mathcal{M}$ is the convex hull of the points $P_1, \cdots, P_n$, $P_{-1}, \cdots, P_{-n}$.
    
    \item If we replace the matrix $M$ in (1) with $M^*_{(n-r)\times n}$ (see Section~\ref{regular}), then we get the Lawrence polytope $\mathcal{P^*}\subseteq\mathbb{R}^{2n-r}$. We use the labels $P_i^*$ for the points generating $\mathcal{P^*}$.  

    \item We further assume that $\mathcal{M}$ is loopless when defining $\mathcal{P}$ and that $\mathcal{M}$ is coloopless when defining $\mathcal{P}^*$, to avoid duplicate columns of the Lawrence matrix. 
\end{enumerate}
\end{definition}
\begin{Rem}
We only need the assumption in (4) for the geometric results in this subsection. In particular, we do not need the assumption for atlases. One can use ``point configurations'' \cite{DRS} to replace polytopes so that the assumption is unnecessary.  
\end{Rem}
\begin{Rem}
Our definition of the Lawrence polytope certainly depends on the matrix we choose. We still say the Lawrence polytope $\mathcal{P}$ (or $\mathcal{P^*}$) of $\mathcal{M}$ for the following two reasons. First, if we fix a total order on the ground set $E$ and fix a reference orientation, then the matrix $M$ is unique up to a multiplication of a matrix in $\mathrm{SL}(r,\mathbb{Z})$ on the left; see \cite{SW}. Hence the resulting Lawrence polytope is also unique in a similar sense. Second, our results involving the Lawrence polytope do not depend on the choice of $M$. 
\end{Rem}
We introduce some basic notions in discrete geometry. 
\begin{definition}
A \emph{simplex} $S$ is the convex hull of some affinely independent points. A \emph{face} of $S$ is a simplex generated by a subset of these points, which could be $S$ or $\emptyset$.
\end{definition}

\begin{definition}\label{tri-diss-def}
Let $\mathcal{P}$ be a polytope of dimension $d$. \begin{enumerate}
    \item If $d+1$ of the vertices of $\mathcal{P}$ form a $d$-dimensional simplex, we call such a simplex a \emph{maximal simplex} of $\mathcal{P}$.
    \item A \emph{dissection} of $\mathcal{P}$ is a collection of maximal simplices of $\mathcal{P}$ such that \begin{enumerate}
        \item[(I)] the union is $\mathcal{P}$, and
        \item[(II)] the relative interiors of any two distinct maximal simplices in the collection are disjoint. 
    \end{enumerate}

    \item If we replace the condition (II) in (2) with the condition (III) that any two distinct maximal simplices in the collection intersect in a common face (which could be empty), then we get a \emph{triangulation}. (See Figure~\ref{tri-dis}.)
\end{enumerate}
\end{definition}

\begin{figure}[ht]
            \centering
            \includegraphics[scale=0.9]{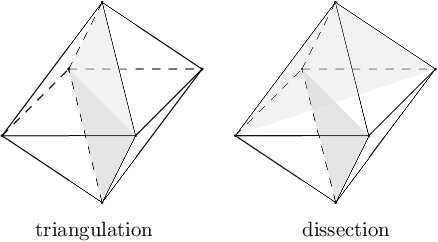}
            \caption{A triangulation and a dissection of an octahedron}
            \label{tri-dis}
\end{figure}  

A triangulation is always a dissection because (III) implies (II). 

The next two theorems build the connection between the geometry of the Lawrence polytopes and the combinatorics of the regular matroid. To state them, we need to label the $2|E|$ arcs of $\mathcal{M}$. Recall that given the matrix $M$, the arcs of $\mathcal{M}$ are the standard unit vectors and their opposites. We denote them by $\overrightarrow{e_1}, \cdots, \overrightarrow{e_n}$ and $\overrightarrow{e_{-1}}, \cdots, \overrightarrow{e_{-n}}$. In particular, $\overrightarrow{e_i}=-\overrightarrow{e_{-i}}$.  
\begin{theorem}\label{3-fold}
We have the following threefold bijections, all of which are denoted by $\chi$. (It should be clear from the context which one we are referring to when we use $\chi$. )
\begin{enumerate}
    \item The Lawrence polytope $\mathcal{P}\subseteq\mathbb{R}^{n+r}$ is an $(n+r-1)$-dimensional polytope whose vertices are exactly the points $P_1, \cdots, P_n, P_{-1}, \cdots, P_{-n}$. Hence we may define a bijection 
    \begin{align*}
    \chi:\{\text{vertices of }\mathcal{P}\} & \to  \{\text{arcs of }\mathcal{M}\} \\
    P_i & \mapsto  \overrightarrow{e_i}
    \end{align*}
    \item The map $\chi$ in (1) induces a bijection
    \begin{align*}
    \chi:\{\text{maximal simplices of }\mathcal{P}\} & \to \{\text{externally oriented bases of }\mathcal{M}\} \\
    \begin{gathered}
        \text{a maximal simplex}\\
        \text{with vertices }\{P_i:i\in I\}
    \end{gathered} & \mapsto \text{the fourientation }\{\chi(P_i):i\in I\}. 
    \end{align*}
    \item The map $\chi$ in (2) induces two bijections
    \begin{align*}
    \chi:\{\text{triangulations of }\mathcal{P}\} & \to \{\text{triangulating external atlases of }\mathcal{M}\} \\
    \begin{gathered}
        \text{a triangulation with} \\
        \text{maximal simplices }\{S_j:j\in J\}
    \end{gathered} & \mapsto \text{the external atlas }\{\chi(S_j):j\in J\}, 
    \end{align*}
    
    and \begin{align*}
    \chi:\{\text{dissections of }\mathcal{P}\} & \to \{\text{dissecting external atlases of }\mathcal{M}\} \\
     \begin{gathered}
         \text{a dissection with}\\\text{maximal simplices }\{S_j:j\in J\}
     \end{gathered} & \mapsto \text{the external atlas }\{\chi(S_j):j\in J\}.
    \end{align*}
    \item In (1), (2) and (3), if we replace the Lawrence polytope $\mathcal{P}$ with $\mathcal{P^*}$, the points $P_i$ with $P_i^*$, $\chi$ with $\chi^*$, and every word ``external'' with ``internal'', then the statement also holds. 
\end{enumerate}
\end{theorem}

Recall that the map $\alpha:\sigma\mapsto\mathcal{A}_\sigma$ is a bijection between triangulating circuit signatures and triangulating external atlases of $\mathcal{M}$. See Section~\ref{regular-sign} for the definition of regular triangulations. 

\begin{theorem}\label{main3}
The restriction of the bijection $\chi^{-1}\circ\alpha$ to the set of acyclic circuit signatures of $\mathcal{M}$ is bijective onto the set of regular triangulations of $\mathcal{P}$. In other words, a circuit signature $\sigma$ is acyclic if and only if the triangulation $\chi^{-1}({A}_\sigma)$ is regular. Dually, the restriction of the bijection $(\chi^*)^{-1}\circ\alpha^*$ to the set of acyclic cocircuit signatures of $\mathcal{M}$ is bijective onto the set of regular triangulations of $\mathcal{P}^*$.
\end{theorem}

We conclude this subsection with Table~\ref{table-summarize}.
\begin{table}[!ht]
\centering
\begin{tabular}{ |c||c|c|c| }
\hline
\begin{tabular}{c}types of dissections of \\ Lawrence polytope $\mathcal{P}$\end{tabular}
 & dissection & triangulation\ & regular triangulation\\
\hline
types of external atlas $\mathcal{A}$ & dissecting & triangulating &  \\
\hline
types of cycle signature $\sigma$ &  & triangulating & acyclic \\
\hline
\end{tabular}
\caption{A summary of the correspondences among dissections of Lawrence polytopes, atlases, and signatures via $\alpha$ and $\chi$. We omit the dual part. A dissecting atlas does not always induce a signature; see Remark~\ref{rem-nonexample}. We do not have an alternative description of the atlases induced by acyclic signatures.}
\label{table-summarize}
\vspace{-8mm}
\end{table}

\subsection{Motivations, root polytopes, and zonotopes}\label{intro-motivation}
We explain how our work is motivated by and related to the work \cite{K}, \cite{KT1}, and \cite{P} on \emph{root polytopes} of \emph{hypergraphs} and the work \cite{BBY} on the BBY bijections. 

\subsubsection{Lawrence polytopes and root polytopes}
Let $G=(V,E)$ be a connected graph without loops, where $V=\{v_i:i\in I\}$ is the vertex set and $E=\{e_j:j\in J\}$ is the edge set. By adding a vertex to the midpoint of each edge of $G$, we obtain a bipartite graph $\text{Bip}(G)$ with vertex classes $V$ and $E$, where we use $e_j$ to label the midpoint of $e_j$ by abusing notation; see Figure~\ref{bip}. We remark that this is a special case of constructing the bipartite graph $\text{Bip}(H)$ associated with a hypergraph $H$; see \cite{K}.

\begin{figure}[ht]
            \centering
            \includegraphics[scale=0.6]{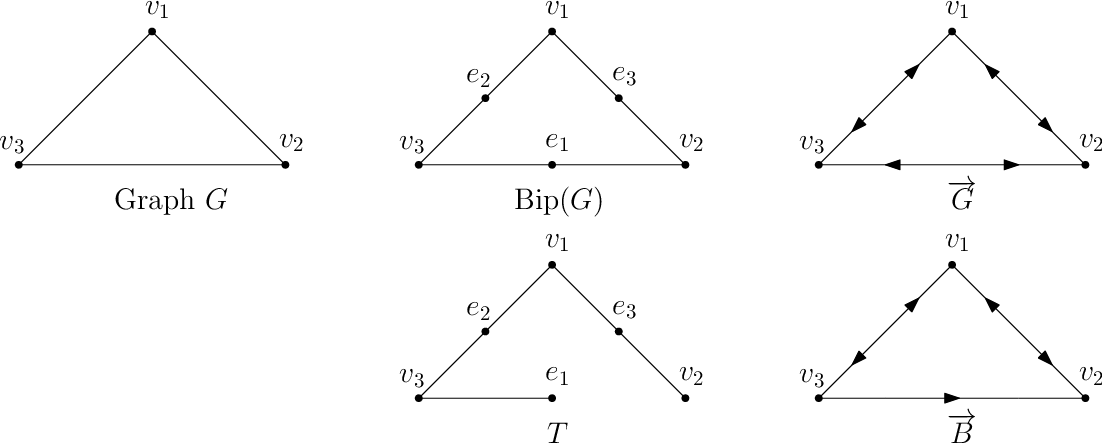}
            \caption{A graph $G$, the bipartite graph $\text{Bip}(G)$ associated to $G$, the fourientation $\protect\overrightarrow{G}$ associated to $G$, a spanning tree $T$ of $\text{Bip}(G)$, and the externally oriented basis $\protect\overrightarrow{B}$ corresponding to $T$.}
            \label{bip}
\end{figure}

Let $\{\mathbf{v}_i:i\in I\}\cup\{\mathbf{e}_j:j\in J\}$ be the coordinate vectors in $\mathbb{R}^{|V|+|E|}$. 
The \emph{root polytope} associated to $\text{Bip}(G)$ is 
\begin{align*}\mathcal{Q}&=\text{ConvexHull}(\mathbf{v}_i-\mathbf{e}_j:v_i \text{ is incident to }e_j \text{ in }G)\\
(&=\text{ConvexHull}(\mathbf{v}_i-\mathbf{e}_j:\{v_i,e_j\}\text{ is an edge of  Bip}(G))).
\end{align*}

The root polytope $\mathcal{Q}$ and the Lawrence polytope $\mathcal{P}$ differ by an invertible linear transformation. Indeed, for every $e_j\in E$ and its two endpoints $v_{i_1},v_{i_2}\in V$, the vectors $\{\mathbf{v}_{i_2}-\mathbf{e}_j, \mathbf{v}_{i_1}-\mathbf{e}_j\}$ can be transformed to $\{\mathbf{e}_j+\mathbf{v}_{i_2}-\mathbf{v}_{i_1}, \mathbf{e}_j\}$ via $\mathbf{e}_j\mapsto\mathbf{v}_{i_1}-\mathbf{e}_j$. Thus the root polytope $\mathcal{Q}$ can be transformed to the Lawrence polytope $\mathcal{P}$ associated to the oriented incidence matrix $M$ of $G$ (ignoring one redundant row of $M$). 

Let $\overrightarrow{G}$ be the fourientation of $G$ where each edge of $G$ is bioriented; see Figure~\ref{bip}. Then we may identify the graph $\text{Bip}(G)$ with the fourientation $\overrightarrow{G}$ as follows. The edges $\{\mathbf{v}_{i_1},\mathbf{e}_j\}$ and $\{\mathbf{v}_{i_2},\mathbf{e}_j\}$ in $\text{Bip}(G)$ are identified with the arcs $(v_{i_1}, v_{i_2})$ and $(v_{i_2}, v_{i_1})$ in $\overrightarrow{G}$ respectively. Via this correspondence, a spanning tree $T\subseteq\text{Bip}(G)$ corresponds to an externally oriented basis $\overrightarrow{B}\subseteq\overrightarrow{G}$. 

Some results in our work can be identified with the ones in \cite{K}, \cite{KT1}, and \cite{P} via the correspondence between $(\text{Bip}(G), \mathcal{Q})$ and $(\overrightarrow{G}, \mathcal{P})$. For example, the following lemma, which characterizes the maximal simplices of $\mathcal{Q}$, is the same as Theorem~\ref{3-fold}(2), which characterizes the maximal simplices of $\mathcal{P}$.
\begin{lemma}\cite[Lemma 12.5]{P} Any maximal simplex of $\mathcal{Q}$ is of the form \[\Delta_T=\text{ConvexHull}(\mathbf{v}_i-\mathbf{e}_j:\{v_i,e_j\}\text{ is an edge of  }T),\] where $T$ is a spanning tree of $\text{Bip}(G)$. 
\end{lemma}

However, \cite[Lemma 12.5]{P} works for hypergraphs while our theorem works for regular matroids. Hence neither of them implies the other. A similar example is the analogy between \cite[Lemma 12.6]{P} and Proposition~\ref{local}(2), both of which characterize when two maximal simplices intersect at a common face. 

K\'alm\'an and T\'othm\'er\'esz \cite[Lemma 7.7]{KT1} showed that the Bernardi tour induces a dissection of the root polytope associated to $\text{Bip}(H)$ for any hypergraph $H$. By specializing their result to graphs, we get that the external atlas $\mathcal{A}_\text{B}$ induces a dissection of the Lawrence polytope $\mathcal{P}$, which gave one of the main motivations of our study. We remark that K\'alm\'an and T\'othm\'er\'esz generalized their result to the root polytopes associated to digraphs \cite[Proposition 5.6]{KT2}. 

Since the BBY bijections and the Bernardi bijections are somewhat similar, we hoped to use the root polytope to encode the two signatures used to define the BBY bijections. The mysterious part was that the dissections of the root polytope associated to $\text{Bip}(G)$ are only related to the external arcs with respect to a given tree, while the BBY bijections are related to both the external and internal arcs. This motivated us to look for a dual object to the root polytope, and the dual object turns out to be the Lawrence polytope $\mathcal{P}^*$.

Lastly, we remark that Galashin, Nenashev, and Postnikov \cite{GNP} constructed a new
combinatorial object called a trianguloid and proved that trianguloids are in bijection with triangulations of the root polytope for any hypergraph. In this paper, we characterize the triangulations of the Lawrence polytope $\mathcal{P}$ in terms of circuit signatures; see Table~\ref{table-summarize} and Theorem~\ref{triangular}. However, we can only apply our results to the root polytopes associated to $\text{Bip}(G)$, where $G$ is a graph instead of a hypergraph. Even for this case, it is unclear how our characterization is related to theirs.

\subsubsection{Lawrence polytopes and zonotopes}
An important way to relate the Lawrence polytope to the BBY bijection is by zonotopal subdivisions. The \emph{zonotope} $Z(M)$ (resp. $Z(M^*)$) is the Minkowski sum of the columns of $M$ (resp. $M^*$). Their subdivisions are used to construct the BBY bijection in \cite[Section 3.4]{BBY}. In particular, every acyclic circuit signature (resp. cocircuit signature) induces a subdivision of $Z(M)$ (resp. $Z(M^*)$) indexed by the bases of $\mathcal{M}$. 

We may view the zonotope $Z(M)$ (resp. $Z(M^*)$) as a section of the Lawrence polytope $\mathcal{P}$ (resp. $\mathcal{P}^*$), which is a special case of \cite[Lemma 3.2]{HRS}. To be precise, denote the columns of $M$ by $M_1,\cdots,M_n$, and recall that the columns of the Lawrence matrix \[\begin{pmatrix} M_{r\times n} & {\bf 0} \\  I_{n\times n} &  I_{n\times n} \end{pmatrix}\] are denoted by $P_1, \cdots, P_n$, $P_{-1}, \cdots, P_{-n}$. Then we have \[Z(M)=\{\sum_{i=1}^n k_iM_i:k_i\in [0,1]\text{ for all }i\},\]
and\[\mathcal{P}=\{\sum_{i=1}^n (k_iP_i+k_{-i}P_{-i}):\sum_{i=1}^n (k_i+k_{-i})=1, k_i\geq 0,k_{-i} \geq 0\text{ for all }i\}.\] 
We take the section $y_1=\cdots=y_n=1/n$ of $\mathcal{P}\subseteq\mathbb{R}^{n+r}$, where $y_1,\cdots,y_n$ denote the last $n$ coordinates of $\mathbb{R}^{n+r}$. A direct computation shows that the zonotope $Z(M)$ is exactly the $n$-th dilate of this section. 

If we restrict a triangulation $\chi^{-1}(\mathcal{A}_\sigma)$ of $\mathcal{P}$ to the (dilated) section of $Z(M)$, we obtain a subdivision of $Z(M)$. This is a special case of the \emph{Cayley Trick}; see \cite{D,HRS,P,Z}. When the signature $\sigma$ is acyclic, it is easy to check that the subdivision of $Z(M)$ is exactly the one induced by $\sigma$ in \cite{BBY}.

\subsection{Open questions}\label{open}

We propose some open questions. 

\subsubsection{Torsors}
Fix a connected graph $G$. By a \emph{torsor structure} on the set $\mathcal{T}(G)$ of spanning trees of $G$, we mean a simply transitive action of the Jacobian group $\text{Jac}(G)$ on $\mathcal{T}(G)$. 

Recall that there is a canonical action of $\text{Jac}(G)$ on the set $\mathcal{R}(G)$ of cycle-cocycle reversal classes of $G$ \cite{B, BBY}. Hence if we have any bijection $\beta$ between $\mathcal{T}(G)$ and $\mathcal{R}(G)$, then we have an induced simply transitive action of $\text{Jac}(G)$ on $\mathcal{T}(G)$:
\begin{equation}\label{eq: bijection and torsor}
\text{Jac}(G)\circlearrowright\mathcal{R}(G) \xlongleftrightarrow{\beta}\mathcal{T}(G).    
\end{equation}

Theorem~\ref{main1} says that a pair of dissecting atlases $(\mathcal{A},\mathcal{A^*})$, at least one of which is triangulating, induces the bijection\begin{align*}
    \overline{f}_{\mathcal{A},\mathcal{A^*}}: \mathcal{T}(G) & \to \mathcal{R}(G) \\
    B & \mapsto [\overrightarrow{B} \cap \overrightarrow{B^*}].
    \end{align*}
Thus the bijection $\overline{f}_{\mathcal{A},\mathcal{A^*}}$ induces a torsor structure on $\mathcal{T}(G)$.

\noindent\textbf{Question 1.} How can we determine whether two bijections $\overline{f}_{\mathcal{A},\mathcal{A^*}}$ induce the same torsor structure on $\mathcal{T}(G)$? In the question, when we compare two group actions, the automorphisms of $\text{Jac}(G)$ are not considered. However, one can also compare the group actions up to automorphisms, which gives a different question. 

This paper focuses on the torsor structures defined by (\ref{eq: bijection and torsor}). In an earlier work \cite{Hol08}, Holroyd \textit{et al.} used the \emph{rotor-routing model} to define a torsor structure on $\mathcal{T}(G)$, called the \emph{rotor-routing torsor structure}. Their definition requires the graph $G$ to be a ribbon graph together with a choice of a basepoint vertex. Several papers, including \cite{CCG,BW,D1,ShW,KLT,GM}, study the rotor-routing torsor structure. We wonder whether the rotor-routing torsor structure fits into our framework. 
    
\noindent\textbf{Question 2.} Given a ribbon graph together with a basepoint vertex, does there exist a pair of dissecting atlases $(\mathcal{A},\mathcal{A^*})$, at least one of which is triangulating, such that the induced torsor structure is the same as the rotor-routing torsor structure? 
 
For any planar graph $G$ with a fixed embedding in the plane, we remark that the answer to Question 2 is yes because a pair of acyclic signatures can be chosen such that the BBY bijection induces the rotor-routing torsor structure; see \cite[Example 1.1.3]{BBY}.

\subsubsection{Triangulating signatures}    

We wonder whether Theorem~\ref{triangular} holds for regular matroids. The precise statement is as follows. 

\noindent\textbf{Question 3.} Let $\sigma$ be a circuit signature of a regular matroid $\mathcal{M}$. Are the following equivalent?
\begin{enumerate}
    \item The signature $\sigma$ is triangulating.
    \item For any three signed circuits $\overrightarrow{C}_1,\overrightarrow{C}_2,\overrightarrow{C}_3\in\sigma$, their sum $\overrightarrow{C}_1+\overrightarrow{C}_2+\overrightarrow{C}_3$
    (as vectors in $\mathbb{Z}^E$) is not zero.
\end{enumerate}

\subsubsection{Combining the two Lawrence polytopes}    
In Theorem~\ref{main1}, we require that\begin{itemize}
    \item the atlas $\mathcal{A}$ is triangulating and the other atlas $\mathcal{A^*}$ is dissecting, or
    \item the atlas $\mathcal{A}$ is dissecting and the other atlas $\mathcal{A^*}$ is triangulating.
\end{itemize}

One can imagine that if we put a stronger condition on one of the atlases, then we could require less on the other. We hope to find one polytope that encodes these information. 

\noindent\textbf{Question 4.} Given a regular matroid, is there a polytope whose subdivisions induce bijections between the bases and the circuit-cocircuit reversal classes of the matroid such that the bijections include the ones in Theorem~\ref{main1}? 

The rest of the paper includes the proofs of all the theorems. 

\section{The Proofs of Theorem~\ref{main1} and Theorem~\ref{main2}}\label{combinatorial results}

We will prove Theorem~\ref{main1} and Theorem~\ref{main2} in this section. 

\subsection{Preliminaries}\label{Pre}

In this subsection, we will introduce some lemmas and notations. Some of them will also be used in other sections.

Let $\mathcal{M}$ be a regular matroid. We start with three lemmas which hold for oriented matroids and hence for regular matroids. In the case of graphs, one can find the later two results in \cite[Lemma 2.4 and Proposition 2.5]{BH}. 

The following lemma is known as the \emph{orthogonality axiom} \cite[Theorem 3.4.3]{BVSWZ}.
\begin{lemma}\label{orientation2}
Let $\overrightarrow{C}$ be a signed circuit and $\overrightarrow{C^*}$ be a signed cocircuit of $\mathcal{M}$. If $C\cap C^*\neq\emptyset$, then there exists an edge on which $\overrightarrow{C}$ and $\overrightarrow{C^*}$ agree and another edge on which $\overrightarrow{C}$ and $\overrightarrow{C^*}$ disagree. 
\end{lemma}

\begin{lemma}\label{exclusive} Let $\overrightarrow{F}$ be a fourientation of $\mathcal{M}$. Then for any potential circuit $\overrightarrow{C}$ and any potential cocircuit $\overrightarrow{C^*}$ of $\overrightarrow{F}$, their underlying edges satisfy $C\cap C^*=\emptyset$. 
\end{lemma}
\begin{proof}
Assume $E_0=C\cap C^*$ is nonempty. Then the edges in $E_0$ must be one-way oriented by $\overrightarrow{C}$ and by $\overrightarrow{C^*}$. This contradicts Lemma~\ref{orientation2}. 
\end{proof}
The following lemma is known as the \emph{3-painting axiom}; see \cite[Theorem 3.4.4]{BVSWZ}. 
\begin{lemma}\label{3-painting} Let $\overrightarrow{F}$ be a fourientation of $\mathcal{M}$ and $\overrightarrow{e}$ be a one-way oriented edge in $\overrightarrow{F}$.
Then $\overrightarrow{e}$ belongs to some potential circuit of $\overrightarrow{F}$ or $\overrightarrow{e}$ belongs to some potential cocircuit of $\overrightarrow{F}$ but not both.
\end{lemma}

We also need the following lemma and definition. Recall that $M$ is a totally unimodular matrix representing the regular matroid $\mathcal{M}$.

\begin{lemma}\cite[Lemma 6.7]{Z}\label{conformal}
(1) Let $\overrightarrow{u}\in\ker_\mathbb{R}(M)$. Then $\overrightarrow{u}$ can be written as a sum of signed circuits with positive coefficients $\sum k_i\overrightarrow{C_i}$ where for each edge $e$ of each $C_i$, the sign of $e$ in $\overrightarrow{C_i}$ agrees with the sign of $e$ in $\overrightarrow{u}$. 

(2) Let $\overrightarrow{u^*}\in\im_\mathbb{R}(M^T)$. Then $\overrightarrow{u^*}$ can be written as a sum of signed cocircuits with positive coefficients $\sum k_i\overrightarrow{C_i^*}$ where for each edge $e$ of each $C_i^*$, the sign of $e$ in $\overrightarrow{C_i^*}$ agrees with the sign of $e$ in $\overrightarrow{u^*}$. 
\end{lemma}

\begin{definition}\label{component}
In Lemma~\ref{conformal}, we call the signed circuit $\overrightarrow{C_i}$ a \emph{component} of $\overrightarrow{u}$ and the signed cocircuit $\overrightarrow{C_i^*}$ a \emph{component} of $\overrightarrow{u^*}$. 
\end{definition}

\begin{Rem}\label{component2}
In Lemma~\ref{conformal}, the linear combination might not be unique. However, if we fix a linear combination, it is clear that the underlying edges of $\overrightarrow{u}$ (i.e., $\{e:e\text{-th entry of }\overrightarrow{u}\text{ is not zero}\}$) is the union of the underlying edges of its components in the linear combination. Also see \cite[Lemma 4.1.1]{BBY} for an integral version of Lemma~\ref{conformal}.   
\end{Rem}

The following lemma is crucial when we deal with circuit-cocircuit reversal classes. One can find a proof in the proof of \cite[Theorem 3.3]{GY} or see \cite[Section 2.1]{D2}.

\begin{lemma}\label{orientation1}
Let $\overrightarrow{O_1}$ and $\overrightarrow{O_2}$ be two orientations in the same circuit-cocircuit reversal class of $\mathcal{M}$. Then $\overrightarrow{O_2}$ can be obtained by reversing disjoint signed circuits and signed cocircuits in $\overrightarrow{O_1}$. 
\end{lemma}

Lastly, we introduce some useful notations here. Recall that $E$ is the ground set of $\mathcal{M}$. Let $E_0$ be a subset of $E$ and $\overrightarrow{F}$ be a fourientation. We denote by $\overrightarrow{F}|_{E_0}$ the fourientation obtained by restricting $\overrightarrow{F}$ to the ground set $E_0$, i.e., $\overrightarrow{F}|_{E_0}=\overrightarrow{F}\cap\{\pm\overrightarrow{e}: e\in E_0\}$. When $E_0$ consists of a single edge $e$, we simply write $\overrightarrow{F}|_e$. In particular, when $e$ is unoriented in $\overrightarrow{F}$, $\overrightarrow{F}|_e=\emptyset$. When $e$ is bioriented in $\overrightarrow{F}$, we write $\overrightarrow{F}|_e=\updownarrow$.

\subsection{Proof of Theorem~\ref{main1}}
We first recall some basic settings. We fix a regular matroid $\mathcal{M}$ with ground set $E$. Let $\mathcal{A}$ be an external atlas and $\mathcal{A^*}$ be an internal atlas, which means for every basis $B$ of $\mathcal{M}$, there exists a unique externally oriented basis $\overrightarrow{B}\in\mathcal{A}$ and a unique internally oriented basis $\overrightarrow{B^*}\in\mathcal{A^*}$. The pair $(\mathcal{A},\mathcal{A^*})$ of atlases induces the following map 
    \begin{align*}
    f_{\mathcal{A},\mathcal{A^*}}:\{\text{bases}\} & \to \{\text{orientations}\} \\
    B & \mapsto \overrightarrow{B} \cap \overrightarrow{B^*}.
    \end{align*}
Let $B_1$ and $B_2$ be two arbitrary bases (not necessarily distinct). Let $\overrightarrow{O_1}$, $\overrightarrow{O_2}$ and $\overrightarrow{F}$, $\overrightarrow{F^*}$ be two orientations and two fourientations given by the following formulas: \[\overrightarrow{O_i}=f_{\mathcal{A},\mathcal{A^*}}(B_i), i\in\{1,2\},\] \[\overrightarrow{F}=\overrightarrow{B_1}\cap(-\overrightarrow{B_2}),\] \[\overrightarrow{F^*}=(\overrightarrow{B_1^*}\cap(-\overrightarrow{B_2^*}))^c.\]

Now we compute the two fourientations $\overrightarrow{F}$ and $\overrightarrow{F^*}$ in terms of $\overrightarrow{O_1}$ and $\overrightarrow{O_2}$, which is summarized in Table~\ref{table1}. For example, when $e\in B_2\backslash B_1$, we have $\overrightarrow{F}|_e=\overrightarrow{O_1}|_e$. This is because $\overrightarrow{F}=\overrightarrow{B_1}\cap(-\overrightarrow{B_2})$, $\overrightarrow{B_2}|_e=\updownarrow$, and $\overrightarrow{B_1}|_e=\overrightarrow{O_1}|_e$ (due to $\overrightarrow{O_1}=\overrightarrow{B_1}\cap\overrightarrow{B_1^*}$). All the other results can be derived similarly. A direct consequence of this table is the following lemma.

\begin{table}[ht]
\centering
\bgroup
\def\arraystretch{1.5}
\begin{tabular}{ |c|c|c|c|c| }

\hline
position of edge $e$ & $B_1\cap B_2$ & $B_1\backslash B_2$ & $B_2\backslash B_1$ & $B_1^c\cap B_2^c$\\
\hline
$\overrightarrow{F}|_{e}$ & $\updownarrow$  & $-\overrightarrow{O_2}$  & $\overrightarrow{O_1}$  & $ \overrightarrow{O_1}\cap(-\overrightarrow{O_2})$ \\
\hline
$\overrightarrow{F^*}|_{e}$ &$(-\overrightarrow{O_1})\cup\overrightarrow{O_2}$ & $-\overrightarrow{O_1}$ &$\overrightarrow{O_2}$ & $\emptyset$\\

\hline
\end{tabular}
\egroup
\caption{The table computes $\protect\overrightarrow{F}$ and $\protect\overrightarrow{F^*}$ in terms of  $\protect\overrightarrow{O_1}$ and $\protect\overrightarrow{O_2}$. The edges $e$ of $\mathcal{M}$ are partitioned into $4$ classes according to whether $e\in B_1$ and whether $e\in B_2$. We view $\protect\overrightarrow{O_1}$ and $\protect\overrightarrow{O_1}$ as sets of arcs of $\mathcal{M}$ so that the union and intersection make sense. We omit ``$|_e$'' after $\protect\overrightarrow{O_i}$'s. E.g., when $e\in B_1^c\cap B_2^c$,  $\protect\overrightarrow{F}|_e=\protect\overrightarrow{O_1}|_e\cap(-\protect\overrightarrow{O_2}|_e)$.}
\label{table1}
\vspace{-8mm}
\end{table}

\begin{lemma}\label{theorem1lemma} Let $E_\rightrightarrows$ be the set of edges where $\overrightarrow{O_1}$ and $\overrightarrow{O_2}$ agree. Let $E_\rightleftarrows=E\backslash E_\rightrightarrows$.

\begin{enumerate}
    \item If $E_0\subseteq E_\rightrightarrows$, then $\overrightarrow{F}|_{E_0}=\overrightarrow{F^*}|_{E_0}$. 
    \item If $E_0\subseteq E_\rightleftarrows$, then $\overrightarrow{O_1}|_{E_0}\subseteq\overrightarrow{F}|_{E_0}$ and $\overrightarrow{F^*}|_{E_0}\subseteq\overrightarrow{O_2}|_{E_0}$. 
\end{enumerate}
\end{lemma}

Theorem~\ref{main1} will be a consequence of the following result. 
\begin{proposition}\label{th1prop}
Let $B_1$ and $B_2$ be two distinct bases of $\mathcal{M}$. If either of the following two assumptions holds, then the orientations $\overrightarrow{O_1}=f_{\mathcal{A},\mathcal{A^*}}(B_1)$ and $\overrightarrow{O_2}=f_{\mathcal{A},\mathcal{A^*}}(B_2)$ are in distinct circuit-cocircuit reversal classes. 

\begin{enumerate}
    \item The external atlas $\mathcal{A}$ is dissecting and the internal atlas $\mathcal{A^*}$ is triangulating. 
    \item The external atlas $\mathcal{A}$ is triangulating and the internal atlas $\mathcal{A^*}$ is dissecting.
\end{enumerate}
\end{proposition}
\begin{proof}
Assume by contradiction that $\overrightarrow{O_1}$ and $\overrightarrow{O_2}$ are in the same circuit-cocircuit reversal class. By Lemma~\ref{orientation1}, there exist disjoint signed circuits $\{\overrightarrow{C_i}\}_{i\in I}$ and signed cocircuits $\{\overrightarrow{C^*_{j}}\}_{j\in J}$ in  $\overrightarrow{O_1}$ by reversing which we may obtain $\overrightarrow{O_2}$. 

(1) Because $\mathcal{A}$ is dissecting, the fourientation $\overrightarrow{F}$ has a potential cocircuit $\overrightarrow{D^*}$. We will show that $\overrightarrow{D^*}$ is also a potential cocircuit of $\overrightarrow{F^*}$, which contradicts that $\mathcal{A^*}$ is triangulating. 

Consider applying Lemma~\ref{theorem1lemma}. Note that $E_\rightleftarrows$ is the disjoint union of  $\{C_i\}_{i\in I}$ and $\{C^*_{j}\}_{j\in J}$. For any $j\in J$, let $E_0=C^*_j$ and apply Lemma~\ref{theorem1lemma}(2). Then we get  $\overrightarrow{F^*}|_{C^*_j}\subseteq\overrightarrow{O_2}|_{C^*_j}=-\overrightarrow{C^*_j}$. By definition, this implies that $-\overrightarrow{C^*_j}$ is a potential cocircuit of $\overrightarrow{F^*}$, which contradicts that $\mathcal{A^*}$ is triangulating. So, $J=\emptyset$. For any $i\in I$, let $E_0=C_i$ and apply Lemma~\ref{theorem1lemma}(2). Then we get $\overrightarrow{C_i}=\overrightarrow{O_1}|_{C_i}\subseteq\overrightarrow{F}|_{C_i}$. This means $\overrightarrow{C_i}$ is a potential circuit of $\overrightarrow{F}$. Because $\overrightarrow{D^*}$ is a potential cocircuit of $\overrightarrow{F}$, by Lemma~\ref{exclusive}, $D^*\cap C_i=\emptyset$. Hence $D^*\subseteq E_\rightrightarrows$. By Lemma~\ref{theorem1lemma}(1), $\overrightarrow{D^*}$ is a potential cocircuit of $\overrightarrow{F^*}$, which gives the desired contradiction. 

(2) This part of the proof is dual to the previous one. To be precise, because $\mathcal{A^*}$ is dissecting, the fourientation $\overrightarrow{F^*}$ has a potential circuit $\overrightarrow{D}$. Then by applying Lemma 2.5, we may prove that $I=\emptyset$, $-\overrightarrow{C^*_j}$ is a potential cocircuit of $\overrightarrow{F^*}$, $D\subseteq E_\rightrightarrows$, and $D$ is a potential circuit of $\overrightarrow{F}$. The last claim contradicts that $\mathcal{A}$ is triangulating.

\end{proof}

\begin{Rem}
If we just want to show the map $f_{\mathcal{A},\mathcal{A^*}}$ is injective under the assumption of Proposition~\ref{th1prop}, the proof is short and works even for oriented matroids. Indeed, assume by contradiction that $\overrightarrow{O_1}=\overrightarrow{O_2}$. Then by Lemma~\ref{theorem1lemma}(1), we get $\overrightarrow{F}=\overrightarrow{F^*}$, which contradicts the definitions of the triangulating atlas and dissecting atlas. 
\end{Rem}

\begin{corollary}[Theorem~\ref{main1}]
 Given a pair of dissecting atlases $(
\mathcal{A},\mathcal{A^*})$ of a regular matroid $\mathcal{M}$, if at least one of the atlases is triangulating, then the map 
    \begin{align*}
    \overline{f}_{\mathcal{A},\mathcal{A^*}}:\{\text{bases of }\mathcal{M}\} & \to \{\text{circuit-cocircuit reversal classes of }\mathcal{M}\} \\
    B & \mapsto [\overrightarrow{B} \cap \overrightarrow{B^*}]
    \end{align*}
is bijective, where $[\overrightarrow{B} \cap \overrightarrow{B^*}]$ denotes the circuit-cocircuit reversal class containing the orientation $\overrightarrow{B} \cap \overrightarrow{B^*}$.  
\end{corollary}
\begin{proof}
By Proposition~\ref{th1prop}, the map is injective. Because the number of circuit-cocircuit reversal classes equals the number of bases \cite[Theorem 3.10]{G2}, the map is bijective. 
\end{proof}

\subsection{Proof of Theorem~\ref{main2}}

We will prove Theorem~\ref{main2} in this section. For the construction of $\varphi_{\mathcal{A},\mathcal{A^*}}$, see Definition~\ref{def-ext}. 
We will prove that $\varphi_{\mathcal{A},\mathcal{A^*}}$ has the following property, which is stronger than bijectivity. 
\begin{definition}
Let $\varphi$ be a map from the set of orientations of $\mathcal{M}$ to the set of subsets of $E$. We say the map $\varphi$ is \emph{tiling} if for any two distinct orientations $\overrightarrow{O_1}$ and $\overrightarrow{O_2}$, there exists an edge $e$ such that $\overrightarrow{O_1}|_e\neq\overrightarrow{O_2}|_e$ and $e\in \varphi(\overrightarrow{O_1})\bigtriangleup \varphi(\overrightarrow{O_2})$. (The symbol $\bigtriangleup$ denotes the symmetric difference of two sets,)
\end{definition}
\begin{Rem}
In \cite[Section 4]{D2}, it is shown that $\varphi$ is tiling if and only if it canonically induces a half-open decomposition of the hypercube $[0,1]^E$, where $[0,1]^E$ is viewed as the set of continuous orientations of $\mathcal{M}$.    
\end{Rem} 
\begin{lemma}\label{lem: tiling implies bijective}
If $\varphi$ is tiling, then $\varphi$ is bijective. 
\end{lemma}
\begin{proof}
The property $e\in \varphi(\overrightarrow{O_1})\bigtriangleup \varphi(\overrightarrow{O_2})$ in the definition implies $\varphi(\overrightarrow{O_1})\neq\varphi(\overrightarrow{O_2})$. Hence $\varphi$ is injective. The domain and codomain of $\varphi$ are equinumerous, so $\varphi$ is bijective. 
\end{proof}

Now we begin to show $\varphi_{\mathcal{A},\mathcal{A^*}}$ is tiling. 
\begin{theorem}\label{strong injectivity}
If either of the following two assumptions holds, then the map $\varphi_{\mathcal{A},\mathcal{A^*}}$ is tiling. 
\begin{enumerate}
    \item The external atlas $\mathcal{A}$ is dissecting and the internal atlas $\mathcal{A^*}$ is triangulating. 
    \item The external atlas $\mathcal{A}$ is triangulating and the internal atlas $\mathcal{A^*}$ is dissecting.
\end{enumerate}  
\end{theorem}
\begin{proof}
Let $\overrightarrow{O_A}$ and $\overrightarrow{O_B}$ be two different orientations of $\mathcal{M}$. Assume by contradiction that the desired edge $e$ does not exist. So, 

\begin{center}
for edges $e\in \varphi_{\mathcal{A},\mathcal{A^*}}(\overrightarrow{O_A})\bigtriangleup \varphi_{\mathcal{A},\mathcal{A^*}}(\overrightarrow{O_B})$, we have $\overrightarrow{O_A}|_e=\overrightarrow{O_B}|_e$.  ($\dagger$)
\end{center}

By the construction of $\varphi_{\mathcal{A},\mathcal{A^*}}$, we can find bases $B_1$ and $B_2$ such that $\overrightarrow{O_A}$ is obtained from reversing disjoint signed circuits  $\{\overrightarrow{C_{1,i}}\}_{i\in I_1}$ and signed cocircuits $\{\overrightarrow{C^*_{1,j}}\}_{j\in J_1}$ in  $\overrightarrow{O_1}:=f_{\mathcal{A},\mathcal{A^*}}(B_1)$,  $\overrightarrow{O_B}$ is obtained from reversing disjoint signed circuits $\{\overrightarrow{C_{2,i}}\}_{i\in I_2}$ and signed cocircuits $\{\overrightarrow{C^*_{2,j}}\}_{j\in J_2}$ in  $\overrightarrow{O_2}:=f_{\mathcal{A},\mathcal{A^*}}(B_2)$, \[\varphi_{\mathcal{A},\mathcal{A^*}}(\overrightarrow{O_A})=(B_1\cup C_1) \backslash C_1^*,\] and \[\varphi_{\mathcal{A},\mathcal{A^*}}(\overrightarrow{O_B})=(B_2\cup C_2) \backslash C_2^*,\]
where $C_k$ is all the underlying edges of $\{\overrightarrow{C_{k,i}}\}_{i\in I_k}$ and $C_k^*$ is all the underlying edges of $\{\overrightarrow{C_{k,i}^*}\}_{i\in J_k}$ for $k=1,2$. We also denote $\overrightarrow{C_k}=\biguplus_{i\in I_k}\overrightarrow{C_{k,i}}$ and $\overrightarrow{C_k^*}=\biguplus_{j\in J_k}\overrightarrow{C_{k,j}^*}$. See Figure~\ref{fig: two phi maps} for a summary.

\begin{figure}[ht]
\centering
\begin{tikzcd}
B_1 \arrow[r, "f_{\mathcal{A},\mathcal{A^*}}", maps to] \arrow[d, "\substack{\text{add }C_1\\ \text{ and remove }C_1^*}"', dashed] & \overrightarrow{O_1} \arrow[d, "\substack{\text{reverse }\overrightarrow{C_1}\\  \text{and } \overrightarrow{C_1^*}}", dashed] & & &B_2 \arrow[r, "f_{\mathcal{A},\mathcal{A^*}}", maps to] \arrow[d, "\substack{\text{add }C_2\\\text{and remove }C_2^*}"', dashed] & \overrightarrow{O_2} \arrow[d, "\substack{\text{reverse }\overrightarrow{C_2} \\ \text{ and } \overrightarrow{C_2^*}}", dashed]\\
\varphi_{\mathcal{A},\mathcal{A^*}}(\overrightarrow{O_A}) &  \overrightarrow{O_A} \arrow[l, "\varphi_{\mathcal{A},\mathcal{A^*}}", maps to] & & &\varphi_{\mathcal{A},\mathcal{A^*}}(\overrightarrow{O_B})  &  \overrightarrow{O_B} \arrow[l, "\varphi_{\mathcal{A},\mathcal{A^*}}", maps to]
\end{tikzcd}
\caption{Some objects and their relations in Definition~\ref{def-ext}.}
\label{fig: two phi maps}
\end{figure}

We still adopt the notations \[\overrightarrow{F}=\overrightarrow{B_1}\cap(-\overrightarrow{B_2}), \overrightarrow{F^*}=(\overrightarrow{B_1^*}\cap(-\overrightarrow{B_2^*}))^c,\] introduced in the previous section.

We compute $\overrightarrow{F}$ and $\overrightarrow{F^*}$ in terms of  $\overrightarrow{C_1},\overrightarrow{C_2},\overrightarrow{C^*_1}$, and $\overrightarrow{C^*_2}$, and the results are summarized in Table~\ref{table2}. The next two paragraphs will explain the table.

\begin{table}
\centering
\bgroup
\def\arraystretch{1.5}
\begin{tabular}{ |c|c|c|c|c| }
\hline
position of $e$ and label & $\alpha$: $B_1\cap B_2$ & $\beta$: $B_1\backslash B_2$ & $\gamma$: $B_2\backslash B_1$ & $\delta$: $B_1^c\cap B_2^c$\\
\hline
\multirow{2}{3cm}{1: $C_1\cap C_2$}  & $\overrightarrow{F}=\updownarrow$  & $\overrightarrow{F}=-\overrightarrow{C_2}$   & $\overrightarrow{F}=\overrightarrow{C_1}$  & $\overrightarrow{F}=\overrightarrow{C_1}\cap(-\overrightarrow{C_2})$\\ 

& & $\overrightarrow{F^*}=-\overrightarrow{C_1}$ & $\overrightarrow{F^*}=\overrightarrow{C_2}$ & $\overrightarrow{F^*}=\emptyset$\\

\hline
\multirow{2}{3cm}{2: $C^*_1\cap C^*_2$}  & $\overrightarrow{F}=\updownarrow$  &  $\overrightarrow{F}=-\overrightarrow{C^*_2}$  & $\overrightarrow{F}=\overrightarrow{C^*_1}$  & \\ 

& $\overrightarrow{F^*}=(-\overrightarrow{C^*_1})\cup\overrightarrow{C^*_2}$ & $\overrightarrow{F^*}=-\overrightarrow{C^*_1}$ & $\overrightarrow{F^*}=\overrightarrow{C^*_2}$ & $\overrightarrow{F^*}=\emptyset$\\

\hline

3: $C^*_1\backslash(C_2\cup C^*_2)$ &  $\overrightarrow{F^*}=-\overrightarrow{C^*_1}$ $_\dagger$& $\overrightarrow{F^*}=-\overrightarrow{C^*_1}$ & $\overrightarrow{F^*}=-\overrightarrow{C^*_1}$ $_\dagger$& $\overrightarrow{F^*}=\emptyset$\\

\hline

4: $C^*_2\backslash(C_1\cup C^*_1)$ &  $\overrightarrow{F^*}=\overrightarrow{C^*_2}$ $_\dagger$ & $\overrightarrow{F^*}=\overrightarrow{C^*_2}$ $_\dagger$ & $\overrightarrow{F^*}=\overrightarrow{C^*_2}$ & $\overrightarrow{F^*}=\emptyset$\\
\hline

5: $C_1\backslash(C_2\cup C^*_2)$ &  $\overrightarrow{F}=\updownarrow$ & $\overrightarrow{F}=\overrightarrow{C_1}$ $_\dagger$& $\overrightarrow{F}=\overrightarrow{C_1}$ & $\overrightarrow{F}=\overrightarrow{C_1}$ $_\dagger$\\
\hline

6: $C_2\backslash(C_1\cup C^*_1)$ &  $\overrightarrow{F}=\updownarrow$ & $\overrightarrow{F}=-\overrightarrow{C_2}$ & $\overrightarrow{F}=-\overrightarrow{C_2}$ $_\dagger$& $\overrightarrow{F}=-\overrightarrow{C_2}$ $_\dagger$\\
\hline

7: $(C_1\cup C_2\cup C^*_1\cup C^*_2)^c$ &  $\overrightarrow{F}=\updownarrow$ & $\overrightarrow{F}=\overrightarrow{F^*}$ $_\dagger$& $\overrightarrow{F}=\overrightarrow{F^*}$ $_\dagger$& $\overrightarrow{F^*}=\emptyset$\\
\hline
\end{tabular}
\egroup
\caption{The computational results used in the proof of Proposition~\ref{strong injectivity}. Two paragraphs in the proof explain the table.}
\label{table2}

\vspace{-8mm}
\end{table}

All the edges $e$ are partitioned into $28$ classes according to whether $e$ is in $B_1$ and/or in $B_2$ (columns), and whether $e$ is in $C_1$, $C_2$, $C_1^*$, and/or $C_2^*$ (rows). Regarding the rows, we start with 4 large classes $(C_1\cup C^*_1) \cap (C_2 \cup C^*_2)$, $(C_1\cup C^*_1)^c \cap (C_2 \cup C^*_2)$, $(C_1\cup C^*_1) \cap (C_2 \cup C^*_2)^c$, and $(C_1\cup C^*_1)^c \cap (C_2 \cup C^*_2)^c$. Using $C_1\cap C^*_1=C_2\cap C^*_2=\emptyset$, we may partition these 4 large classes into small classes. However, two items $C_1\cap C_2^*$ and $C_1^*\cap C_2$ are missing in the table. This is because they are empty. Indeed, if one of them, say $C_1\cap C_2^*$, is not empty, then by Lemma~\ref{orientation2}, there exists an edge $e\in C_1\cap C_2^*$ such that $\overrightarrow{C_1}|_e\neq\overrightarrow{C_2^*}|_e$. This implies $\overrightarrow{O_A}|_e\neq\overrightarrow{O_B}|_e$. By the definition of $\varphi_{\mathcal{A},\mathcal{A^*}}$, $e\in C_1\cap C_2^*$ implies $e\in\varphi_{\mathcal{A},\mathcal{A^*}}(\overrightarrow{O_A})\bigtriangleup \varphi_{\mathcal{A},\mathcal{A^*}}(\overrightarrow{O_B})$. By the assumption ($\dagger$), we have $\overrightarrow{O_A}|_e=\overrightarrow{O_B}|_e$, which gives the contradiction. So, the rows of the table cover all the cases.

In the table, we view $\overrightarrow{C_1},\overrightarrow{C_2},\overrightarrow{C^*_1}$, and $\overrightarrow{C^*_2}$ as sets of arcs, so the union and intersection make sense. We omit ``$|_e$'' as in Table~\ref{table1}. We only give the useful results, so some fourientations in some cells are not given. The computation is straightforward. If there is no $\dagger$ in the cell, then we can get the result by making use of Table~\ref{table1} where $\overrightarrow{O_k}$ can be replaced by $\overrightarrow{C_k}$ when $e\in C_k$ and by $\overrightarrow{C_k^*}$ when $e\in C_k^*$, for $k=1,2$. If there is a $\dagger$ in the cell, then the computation makes use of the assumption ($\dagger$). For example, for cells $3\alpha$ and $3\gamma$, since $e\in C_1^*$ and $e\in B_2$, we have $e\notin\varphi_{\mathcal{A},\mathcal{A^*}}(O_A)$ and $e\in\varphi_{\mathcal{A},\mathcal{A^*}}(O_B)$. By ($\dagger$), we have $\overrightarrow{O_A}|_e=\overrightarrow{O_B}|_e$, and hence $\overrightarrow{C_1^*}|_e=\overrightarrow{O_1}|_e=-\overrightarrow{O_2}|_e$. Combining this formula and Table~\ref{table1}, we obtain the formulas in $3\alpha$ and $3\gamma$. Similarly, we obtain the formulas in cells with $\dagger$ in rows 4,5, and 6. For cells $7\beta$ and $7\gamma$, we still have $\overrightarrow{O_A}|_e=\overrightarrow{O_B}|_e$ due to ($\dagger$), which implies $\overrightarrow{O_1}|_e=-\overrightarrow{O_2}|_e$. Then by Table~\ref{table1}, we get $\overrightarrow{F}=\overrightarrow{F^*}$. 

Now we use the table to prove two claims.  

\textbf{Claim 1}: if $\overrightarrow{C_1}-\overrightarrow{C_2}\neq 0$, then each of its components (see Definition~\ref{component}) is a potential circuit of $\overrightarrow{F}$.    

By Lemma~\ref{conformal}, it suffices to check that for any arc $\overrightarrow{e}\in\overrightarrow{C_1}\cup(-\overrightarrow{C_2})$ that is not canceled in $\overrightarrow{C_1}-\overrightarrow{C_2}$, we have $\overrightarrow{F}|_e=\overrightarrow{e}$ or $\overrightarrow{F}$ is bioriented. This follows directly from the rows 1, 5, and 6 in Table~\ref{table2} ($\overrightarrow{C_1}|_e=-\overrightarrow{C_2}|_e$ in row 1). 

Similarly, we can prove the other claim. 

\textbf{Claim 2}: if $\overrightarrow{C_2^*}-\overrightarrow{C_1^*}\neq 0$, then each of its components is a potential cocircuit of $\overrightarrow{F^*}$. 

We are ready to complete the proof. 

When $B_1=B_2$, by definition $\overrightarrow{F}$ has no potential circuit, and $\overrightarrow{F^*}$ has no potential cocircuit. By Claim 1 and Claim 2, $\overrightarrow{C_1}=\overrightarrow{C_2}$ and $\overrightarrow{C_2^*}=\overrightarrow{C_1^*}$, which implies $\overrightarrow{O_A}=\overrightarrow{O_B}$. Contradiction. 

From now on we assume $B_1\neq B_2$. We will apply the dissecting and triangulating conditions (1) or (2) to get contradictions. 

(1) Because $\mathcal{A}$ is dissecting and $\mathcal{A^*}$ is triangulating, there exists a potential cocircuit $\overrightarrow{D^*}$ of $\overrightarrow{F}$, and there is no potential cocircuit of $\overrightarrow{F^*}$. The later one implies that $\overrightarrow{C_2^*}=\overrightarrow{C_1^*}$ by Claim 2. So, rows 3 and 4 in Table~\ref{table2} can be ignored in this case, and in cells $2\alpha$, $2\beta$, and $2\gamma$, $\overrightarrow{F}=\overrightarrow{F^*}$. 

Now we claim that the potential cocircuit $\overrightarrow{D^*}$ of $\overrightarrow{F}$ is also a potential cocircuit of $\overrightarrow{F^*}$, which gives the contradiction. Indeed, on one hand, for edges $e$ in rows 5 and 6, and for edges $e$ in row 1 such that $\overrightarrow{C_1}|_e=-\overrightarrow{C_2}|_e$, they are exactly the underlying edges of $\overrightarrow{C_1}-\overrightarrow{C_2}$, and hence by Claim 1, Lemma~\ref{exclusive}, and Remark~\ref{component2}, as a potential cocircuit of $\overrightarrow{F}$, $\overrightarrow{D^*}$ does not use these edges at all. On the other hand, for the remaining edges $e$, which are those in rows 2 and 7, and in row 1 such that $\overrightarrow{C_1}|_e=\overrightarrow{C_2}|_e$, we have either $\overrightarrow{F}|_e=\overrightarrow{F^*}|_e$ or $\overrightarrow{F^*}|_e=\emptyset$. So, $\overrightarrow{D^*}$ is also a potential cocircuit of $\overrightarrow{F^*}$.

(2) This part can be proved by a similar argument. 
\end{proof}

\begin{corollary}[Theorem~\ref{main2}(1)]
The map $\varphi_{\mathcal{A},\mathcal{A^*}}$ is bijective.    
\end{corollary}
\begin{proof}
This is a direct consequence of Lemma~\ref{lem: tiling implies bijective} and Theorem~\ref{strong injectivity}.    
\end{proof}

Now we prove the rest of Theorem~\ref{main2}.

\begin{theorem}[Theorem~\ref{main2}(2)(3)]\label{main2(2)(3)}
Under either of the assumptions of Proposition~\ref{strong injectivity} on the atlases $\mathcal{A}$ and $\mathcal{A^*}$, we have the following properties of $\varphi_{\mathcal{A},\mathcal{A^*}}$.

(1) The image of the independent sets of $\mathcal{M}$ under the bijection $\varphi^{-1}_{\mathcal{A},\mathcal{A^*}}$ is a representative set of the circuit reversal classes of $\mathcal{M}$. 

(2)  The image of the spanning sets of $\mathcal{M}$ under the bijection $\varphi^{-1}_{\mathcal{A},\mathcal{A^*}}$ is a representative set of the cocircuit reversal classes of $\mathcal{M}$. 
\end{theorem}
\begin{proof}
Recall that the map
    \begin{align*}
    \varphi_{\mathcal{A},\mathcal{A^*}}:\{\text{orientations of }\mathcal{M}\} & \to \{\text{subsets of }E\} \\
    \overrightarrow{O} & \mapsto (B\cup \biguplus_{i\in I}C_i)\backslash \biguplus_{j\in J}C_j^*
    \end{align*}
is a bijection, where $B$ is the unique basis such that $f_{\mathcal{A},\mathcal{A^*}}(B)\in [\overrightarrow{O}]$, and the orientations $f_{\mathcal{A},\mathcal{A^*}}(B)$ and $\overrightarrow{O}$ differ by  disjoint signed circuits $\{\overrightarrow{C_i}\}_{i\in I}$ and cocircuits $\{\overrightarrow{C_j^*}\}_{j\in J}$.

Let $A=\varphi_{\mathcal{A},\mathcal{A^*}}(\overrightarrow{O})$. Then $A$ is an independent set $\Leftrightarrow$ $I=\emptyset$ (due to Lemma~\ref{orientation1}) $\Leftrightarrow$ The orientations $\overrightarrow{O}$ and $f_{\mathcal{A},\mathcal{A^*}}(B)$ are in the same cocircuit reversal class. 

Because the set $\{f_{\mathcal{A},\mathcal{A^*}}(B):B\text{ is a basis}\}$ is a representative set of the circuit-cocircuit reversal classes (Theorem~\ref{main1}), the set  $\{\varphi^{-1}_{\mathcal{A},\mathcal{A^*}}(A):A\text{ is independent}\}$ is a representative set of the circuit-reversal classes. 

This proves (1). Similarly, (2) also holds. 
\end{proof}

\section{Signatures, the BBY bijection, and the Bernardi bijection}\label{signature}
In this section we will use our theory to recover and generalize the work in \cite{BBY}, \cite{D2}, and \cite{Bernardi}. In particular, Theorem~\ref{trig-intro}, Proposition~\ref{prop: triangulating sign rep}, Theorem~\ref{main1-sign}, Theorem~\ref{main2-sign}, and Theorem~\ref{triangular} will be proved. 

To do this, we will build the connection between circuit signatures (resp. cocircuit signatures) and external atlases (resp. internal atlases) of the regular matroid $\mathcal{M}$. We will also see how the BBY bijection (resp. the extended BBY bijection) and the Bernardi bijection become a special case of Theorem~\ref{main1} (resp. Theorem~\ref{main2}). In particular, the acyclic signatures used to define the BBY bijection will be generalized to \emph{triangulating signatures}.

\subsection{Signatures and atlases}

Recall that given a circuit signature $\sigma$, we may construct the external atlas $\mathcal{A_\sigma}$ from $\sigma$ such that for each externally oriented basis $\overrightarrow{B}\in\mathcal{A_\sigma}$, each external arc $\overrightarrow{e}\in\overrightarrow{B}$ is oriented according to the orientation of the fundamental circuit $C(B,e)$ in $\sigma$. Similarly, we may construct the internal atlas $\mathcal{A}^*_{\sigma^*}$.

We now show that all the triangulating atlases can be obtained in this way. Moreover, they must come from triangulating signatures. The following lemma is trivial but useful. 
\begin{lemma}\label{funda}
Every circuit of $\mathcal{M}$ is a fundamental circuit $C(B,e)$ for some basis $B$ and some edge $e$. Dually, every cocircuit of $\mathcal{M}$ is a fundamental cocircuit $C^*(B,e)$ for some basis $B$ and some edge $e$. 
\end{lemma}

\begin{theorem}[Theorem~\ref{trig-intro}]\label{trig}
\begin{enumerate}
    \item The map  \begin{align*}
    \alpha:\{\text{triangulating circuit sig. of }\mathcal{M}\} & \to \{\text{triangulating external atlases of }\mathcal{M}\} \\
    \sigma & \mapsto \mathcal{A_\sigma}
    \end{align*}
    is a bijection.
    
    \item The map  \begin{align*}
    \alpha^*:\{\text{triangulating cocircuit sig. of }\mathcal{M}\} & \to \{\text{triangulating internal atlases of }\mathcal{M}\} \\
    \sigma^* & \mapsto \mathcal{A}^*_{\sigma^*}
    \end{align*}
    is a bijection.
\end{enumerate}
\end{theorem}

\begin{proof}
We only prove (1) because the same method can be used to prove (2). 

First we check the atlas $\mathcal{A_\sigma}$ is triangulating when $\sigma$ is triangulating. Assume by contradiction that there exist distinct bases $B_1$ and $B_2$ such that $\overrightarrow{B_1}\cap(-\overrightarrow{B_2})$ has a potential circuit $\overrightarrow{C}$. Then $\overrightarrow{C}\subseteq \overrightarrow{B_1}$ and $-\overrightarrow{C}\subseteq \overrightarrow{B_2}$. By the definition of $\sigma$ being triangulating, $\overrightarrow{C}\in \sigma$ and $-\overrightarrow{C}\in \sigma$, which gives the contradiction. 

The map $\alpha$ is injective. Indeed, given two different signatures $\sigma_1$ and $\sigma_2$, there exists a signed circuit $\overrightarrow{C}$ such that $\overrightarrow{C}\in\sigma_1$ and $-\overrightarrow{C}\in\sigma_2$. By Lemma~\ref{funda}, $C$ is a fundamental circuit $C(B, e)$. Then the two externally oriented bases associated to $B$ in $\mathcal{A}_{\sigma_1}$ and in $\mathcal{A}_{\sigma_2}$ have different signs on $e$. 

The map $\alpha$ is surjective. Given a triangulating external atlas $\mathcal{A}$, we need to find a triangulating signature $\sigma$ such that $\mathcal{A}=\mathcal{A}_\sigma$. By Lemma~\ref{funda}, any circuit $C$ is a fundamental circuit $C(B, e)$. Then we define $\sigma(C)$ to be the signed circuit $\overrightarrow{C}$ in $\overrightarrow{B}\in\mathcal{A}$. This is well-defined. Indeed, if from two different bases $B_1$ and $B_2$ we get two opposite signed circuits $\overrightarrow{C}$ and $-\overrightarrow{C}$, then $\overrightarrow{C}\subseteq\overrightarrow{B_1}$ and $-\overrightarrow{C}\subseteq\overrightarrow{B_2}$. Hence $\overrightarrow{C}\subseteq\overrightarrow{B_1}\cap(-\overrightarrow{B_2})$, which contradicts $\mathcal{A}$ being triangulating. It is obvious that $\mathcal{A}=\mathcal{A}_\sigma$. It remains to show that $\sigma$ is triangulating. For any $\overrightarrow{B_1}\in\mathcal{A_\sigma}$ and any signed circuit $\overrightarrow{C}\subseteq\overrightarrow{B_1}$, we need to show $\overrightarrow{C}\in \sigma$. If $C$ is a fundamental circuit with respect to $B_1$, then it is done. Otherwise, by Lemma~\ref{funda}, $C=C(B_2,e)$ for some other basis $B_2$. Then either $\overrightarrow{C}\subseteq\overrightarrow{B_2}$ or $-\overrightarrow{C}\subseteq\overrightarrow{B_2}$. The second option is impossible because $\overrightarrow{B_1}\cap(-\overrightarrow{B_2})$ does not contain any signed circuit. Thus $\overrightarrow{C}\subseteq\overrightarrow{B_2}$ and hence $\overrightarrow{C}\in\sigma$. 
\end{proof}

\subsection{Acyclic signatures} In this subsection, we prove acyclic signatures are triangulating. This is essentially \cite[Lemma~2.10 and Lemma~2.9]{D2}. For readers' convenience, we give a proof here, which consists of two lemmas. 

Let $\overrightarrow{e}$ be an arc. We denote by $\overrightarrow{C}(B,\overrightarrow{e})$ the fundamental circuit oriented according to $\overrightarrow{e}$ when $e\notin B$, and denote by $\overrightarrow{C^*}(B,\overrightarrow{e})$ the fundamental cocircuit oriented according to $\overrightarrow{e}$ when $e\in B$. 

\begin{lemma}\label{fundamental}
Fix a basis $B$ of $\mathcal{M}$. 

(1) For any signed circuit $\overrightarrow{C}$, \[\overrightarrow{C}=\sum_{e\notin B,\overrightarrow{e}\in\overrightarrow{C}}\overrightarrow{C}(B,\overrightarrow{e}).\]

(2) For any signed cocircuit $\overrightarrow{C^*}$, \[\overrightarrow{C^*}=\sum_{e\in B,\overrightarrow{e}\in\overrightarrow{C^*}}\overrightarrow{C^*}(B,\overrightarrow{e}).\]  
\end{lemma}
\begin{proof}
We only prove (1) since the method works for (2). 

Note that the set of signed fundamental circuits with respect to $B$ form a basis of $\ker_\mathbb{R}(M)$ (choose an arbitrary orientation for each circuit) \cite{SW}. Hence we can write $\overrightarrow{C}\in\ker_\mathbb{R}(M)$ as a linear combination of these fundamental circuits with real coefficients:
\[\overrightarrow{C}=\sum_{e\notin B}k_e\overrightarrow{C}(B,\overrightarrow{e}).\]
By comparing the coefficients of $e\notin B$ in both sides, we get the desired formula. 
\end{proof}

\begin{lemma}\label{acyclic-tri}
Let $\sigma$ be an acyclic circuit signature and $\sigma^*$ be an 
acyclic cocircuit signature. Then $\sigma$ and $\sigma^*$ are triangulating. (Equivalently, $\mathcal{A}_\sigma$ and $\mathcal{A}^*_{\sigma^*}$ are triangulating atlases.) 
\end{lemma}
\begin{proof}
We only give the proof for $\sigma$. 
By definition, for any $\overrightarrow{B}\in\mathcal{A_\sigma}$ and any signed circuit $\overrightarrow{C}\subseteq\overrightarrow{B}$, we need to show $\overrightarrow{C}\in \sigma$. 

By Lemma~\ref{fundamental}, \[\overrightarrow{C}=\sum_{e\notin B,\overrightarrow{e}\in\overrightarrow{C}}\overrightarrow{C}(B,\overrightarrow{e}).\]

Since $\overrightarrow{B}\in\mathcal{A_\sigma}$, every signed circuit in the right-hand side is in $\sigma$. By the definition of $\sigma$ being acyclic, we have $\overrightarrow{C}\in\sigma$. So, $\sigma$ is triangulating.
\end{proof}

There exists a triangulating circuit signature that is not acyclic. See Section~\ref{nonexample} for an example together with a nice description of the triangulating \emph{cycle} signatures of \emph{graphs}. 

\subsection{The BBY bijection and compatible orientations} 
Given a pair $(\sigma,\sigma^*)$ of triangulating signatures, we write \[\text{BBY}_{\sigma,\sigma^*}=f_{\mathcal{A_\sigma},\mathcal{A^*_{\sigma^*}}}\text{ and }\varphi_{\sigma,\sigma^*}=\varphi_{\mathcal{A_\sigma},\mathcal{A^*_{\sigma^*}}}.\]
They are exactly the BBY bijection in \cite{BBY} and the extended BBY bijection in \cite{D2} when the two signatures are acyclic. By the results in the previous two subsections, we may apply Theorem~\ref{main1} and Theorem~\ref{main2} to these two maps and hence generalize the counterpart results in \cite{BBY} and \cite{D2}. 
Compared with atlases, signatures allow us to talk about compatible orientations; see Section~\ref{sign} for the definition. The maps $\text{BBY}_{\sigma,\sigma^*}$ and $\varphi_{\sigma,\sigma^*}$ are proved to be bijective onto compatible orientations in addition to orientation classes in \cite{BBY,D2}. Here is an example.  
\begin{theorem}\label{BBYTh}\cite[Theorem 1.3.1]{BBY}
Suppose $\sigma$ and $\sigma^*$ are \emph{acyclic} signatures of $\mathcal{M}$.
\begin{enumerate}
    \item The map $\text{BBY}_{\sigma,\sigma^*}$ is a bijection between the bases of $\mathcal{M}$ and the $(\sigma,\sigma^*)$-compatible orientations of $\mathcal{M}$. 
    \item The set of $(\sigma, \sigma^*)$-compatible orientations is a representative set of the circuit-cocircuit reversal classes of $\mathcal{M}$.
\end{enumerate}
\end{theorem}

We will also generalize these results by reformulating Theorem~\ref{main1} and Theorem~\ref{main2} in terms of signatures and compatible orientations. We first prove a lemma. 

\begin{lemma}\label{compatibleimage}
Suppose $\sigma$ and $\sigma^*$ are \emph{triangulating} signatures of $\mathcal{M}$. Then for any basis $B$, the orientation $\text{BBY}_{\sigma,\sigma^*}(B)$ is $(\sigma, \sigma^*)$-compatible.
\end{lemma}
\begin{proof}
For any signed circuit $\overrightarrow{C}\subseteq\text{BBY}_{\sigma,\sigma^*}(B)=\overrightarrow{B}\cap\overrightarrow{B^*}$, where $\overrightarrow{B}\in\mathcal{A_\sigma}$ and $\overrightarrow{B^*}\in\mathcal{A^*_{\sigma^*}}$, we have $\overrightarrow{C}\subseteq\overrightarrow{B}$, and hence $\overrightarrow{C}$ is in the signature $\sigma$ because $\sigma$ is triangulating. Similarly, for any signed cocircuit in the orientation $\text{BBY}_{\sigma,\sigma^*}(B)$, it is in $\sigma^*$. So, the orientation $\text{BBY}_{\sigma,\sigma^*}(B)$ is $(\sigma, \sigma^*)$-compatible. 
\end{proof}

Now we generalize Theorem~\ref{BBYTh}. 
\begin{theorem}\label{tri-theorem}
Suppose $\sigma$ and $\sigma^*$ are \emph{triangulating} signatures of $\mathcal{M}$. \begin{enumerate}
    \item (Theorem~\ref{main1-sign}) The map $\text{BBY}_{\sigma,\sigma^*}$ is a bijection between the bases of $\mathcal{M}$ and the $(\sigma,\sigma^*)$-compatible orientations of $\mathcal{M}$. 
    \item (Proposition~\ref{prop: triangulating sign rep}(1)) The set of $(\sigma, \sigma^*)$-compatible orientations is a representative set of the circuit-cocircuit reversal classes of $\mathcal{M}$.
\end{enumerate}
\end{theorem}

\begin{proof}
It is a direct consequence of the following three facts. \begin{itemize}
    \item By Theorem~\ref{main1}, the image of $\text{BBY}_{\sigma,\sigma^*}$ forms a representative set of the circuit-cocircuit reversal classes.
    \item By Lemma~\ref{compatibleimage}, the image of $\text{BBY}_{\sigma,\sigma^*}$ is contained in the set of $(\sigma, \sigma^*)$-compatible orientations.
    \item By Lemma~\ref{orientation1}, each circuit-cocircuit reversal class contains at most one $(\sigma, \sigma^*)$-compatible orientation. 
\end{itemize}
\end{proof}

Theorem~5.2 in \cite{D2} says that the following result on the extended BBY bijection holds for acyclic signatures. Now we prove that it holds for triangulating signatures. 
\begin{theorem}[Theorem~\ref{main2-sign}]\label{tri-theorem2}

Suppose $\sigma$ and $\sigma^*$ are \emph{triangulating} signatures of a regular matroid $\mathcal{M}$ with ground set $E$.  

(1) The map 
\begin{align*}
\varphi_{\sigma,\sigma^*}:\{\text{orientations of }\mathcal{M}\} & \to \{\text{subsets of } E\} \\
\overrightarrow{O} & \mapsto (\text{BBY}_{\sigma,\sigma^*}^{-1}(\overrightarrow{O^{cp}})\cup \biguplus_{i\in I}C_i)\backslash \biguplus_{j\in J}C_j^*
\end{align*}
is a bijection, where $\overrightarrow{O^{cp}}$ is  the (unique) ($\sigma,\sigma^*$)-compatible orientation obtained by reversing disjoint signed circuits $\{\overrightarrow{C_i}\}_{i\in I}$ and signed cocircuits $\{\overrightarrow{C_j^*}\}_{j\in J}$ in $\overrightarrow{O}$.

(2) The map $\varphi_{\sigma,\sigma^*}$ specializes to the bijection
\begin{align*}
\varphi_{\sigma,\sigma^*}: \{\sigma\text{-compatible orientations}\} & \to \{\text{independent sets}\} \\
\overrightarrow{O} & \mapsto \text{BBY}_{\sigma,\sigma^*}^{-1}(\overrightarrow{O^{cp}})\backslash \biguplus_{j\in J}C_j^*.
\end{align*}

(3) The map $\varphi_{\sigma,\sigma^*}$ specializes to the bijection
\begin{align*}
\varphi_{\sigma,\sigma^*}:\{\sigma^*\text{-compatible orientations}\} & \to \{\text{spanning sets}\} \\
\overrightarrow{O} & \mapsto \text{BBY}_{\sigma,\sigma^*}^{-1}(\overrightarrow{O^{cp}})\cup \biguplus_{i\in I}C_i.
\end{align*}
\end{theorem}

\begin{proof}
(1) This is a direct consequence of Theorem~\ref{main2}(1) and Theorem~\ref{tri-theorem}. 

(2) Let $A=\varphi_{\sigma,\sigma^*}(\overrightarrow{O})$. Then $A$ is an independent set $\Leftrightarrow$ $I=\emptyset$ (by Lemma~\ref{orientation1}) $\Leftrightarrow$ $\overrightarrow{O}$ is $\sigma$-compatible.

(3) The proof is similar to the one of (2). 
\end{proof}

The following result is a direct consequence of Theorem~\ref{main2} and Theorem~\ref{tri-theorem2}. 

\begin{corollary}[Proposition~\ref{prop: triangulating sign rep}(2)(3)]
For a \emph{triangulating} signature $\sigma$ (resp. $\sigma^*$), the $\sigma$-compatible (resp. $\sigma^*$-compatible) orientations form representatives for circuit reversal classes (resp. cocircuit reversal classes). 
\end{corollary}

\begin{Rem}
We can further generalize Theorem~\ref{tri-theorem2}(2)(3) a bit. From the proof of Theorem~\ref{tri-theorem2}(2) (including the preceding lemmas), it is clear that if $\sigma$ is a triangulating signature and $\mathcal{A^*}$ is a dissecting atlas, then $\varphi_{\mathcal{A_\sigma},\mathcal{A^*}}$ specializes to a bijection between $\{\sigma\text{-compatible orientations}\}$ and $\{\text{independent sets}\}$. The dual statement also holds. 
\end{Rem}

So far, we have proved every claim in Section~\ref{sign} except Theorem~\ref{triangular}, which we will prove next.

\subsection{Triangulating cycle signatures of graphs}\label{nonexample}

For graphs, we have a nice description (Theorem~\ref{triangular}) for the triangulating cycle signatures. We will prove the result and use it to check an example where a triangulating cycle signature is not acyclic. 

Let $G$ be a graph where multiple edges are allowed (loops are of no interest here). By cycles of $G$, we mean simple cycles. When we add cycles, we view them as vectors in $\mathbb{Z}^E$. 

We start with a basic lemma. We cannot find a reference, so we prove it briefly. 

\begin{lemma}\label{theta}
If the sum of two directed cycles $\overrightarrow{C_1}$ and $\overrightarrow{C_2}$ of $G$ is a directed cycle, then their common edges $C_1\cap C_2$ form a path (which is directed in opposite ways in the two directed cycles). 
\end{lemma}
\begin{proof}
Clearly $C_1\cap C_2$ contains a path. Take a maximal path and consider its two endpoints $v_1$ and $v_2$. We put a chip $c$ at $v_1$ and move $c$ along $\overrightarrow{C_1}$. Without loss of generality, we may assume $c$ leaves the path and certainly leaves $C_2$. We claim that the next place where $c$ reaches $C_2$ is $v_2$, which finishes the proof of the lemma. Indeed, if $c$ reaches a common vertex of $C_1$ and $C_2$ other than $v_2$, then we can move $c$ back to $v_1$ along $\overrightarrow{C_2}$, and hence the route of $c$ forms a directed cycle which is strictly contained in $\overrightarrow{C_1}+\overrightarrow{C_2}$. This contradicts that $\overrightarrow{C_1}+\overrightarrow{C_2}$ is a cycle. 
\end{proof}

Our target is to prove the following result. The proof needs a technical lemma which we will state and prove right after the result. Gleb Nenashev gave the proof (including the technical lemma). 
\begin{theorem}[Theorem~\ref{triangular}]\label{tri-sign}\cite{Nenashev}
Let $\sigma$ be a \emph{cycle} signature of a \emph{graph} $G$. Then the following are equivalent. 
\begin{enumerate}
    \item $\sigma$ is triangulating. 
    \item For any three directed cycles in $\sigma$, the sum is not zero.
    \item For any two directed cycles in $\sigma$, if their sum is a cycle, then the sum is in $\sigma$. 
\end{enumerate}
\end{theorem}
\begin{proof}
The equivalence of (2) and (3) is trivial. Without loss of generality, we may assume $G$ is connected. 

Then we prove (1) implies (3). Denote the two directed cycles by $\overrightarrow{C_1}$ and $\overrightarrow{C_2}$, and their sum by $\overrightarrow{C}$. By Lemma~\ref{theta}, $C_1\cap C_2$ is a path $P$. Hence we can get a forest from $C_1\cup C_2$ by removing one edge in $C_1\backslash P$ and one edge in $C_2\backslash P$. We extend the forest to a spanning tree $B$ of $G$. Then $C_1$ and $C_2$ are both fundamental cycles of $G$ with respect to $B$. Consider the external atlas $\overrightarrow{B}\in \mathcal{A}_\sigma$. Because $C_1, C_2\in\sigma$, we have $\overrightarrow{C_1}, \overrightarrow{C_2}\subseteq \overrightarrow{B}$, and hence $\overrightarrow{C}\subseteq\overrightarrow{B}$. Because $\sigma$ is triangulating, $\overrightarrow{C}\in\sigma$.  

The difficult part is (3) implies (1). For any $\overrightarrow{B}\in\mathcal{A_\sigma}$ and any signed circuit $\overrightarrow{C}\subseteq\overrightarrow{B}$, we want to show $\overrightarrow{C}\in\sigma$. By Lemma~\ref{Gleb}, we can write $\overrightarrow{C}$ as the sum of directed fundamental cycles with a complete parenthesization such that each time we add two directed cycles up, the sum is always a directed cycle. Because $\overrightarrow{C}\subseteq\overrightarrow{B}$, all the directed fundamental cycles in the summation are in $\sigma$. Due to (3), $\overrightarrow{C}\in\sigma$. 
\end{proof}

Let $B$ be a spanning tree of a connected graph $G$ and $\overrightarrow{C}$ be a directed cycle. Denote by $\overrightarrow{e_1},\cdots,\overrightarrow{e_m}$ the external arcs (with respect to $B$) that appear in $\overrightarrow{C}$ in order (with an arbitrary start). By Lemma~\ref{fundamental}, \[\overrightarrow{C}=\sum_{i=1}^m\overrightarrow{C}(B,\overrightarrow{e_i}).\]

\begin{lemma}\label{Gleb} \cite{Nenashev} The summation above can be completely parenthesized such that during the summation the sum of two terms is always a directed cycle.

(e.g., $\overrightarrow{C}=(\overrightarrow{C_1}+\overrightarrow{C_2})+(\overrightarrow{C_3}+(\overrightarrow{C_4}+\overrightarrow{C_5}))$ is completely parenthesized, and we hope $\overrightarrow{C_1}+\overrightarrow{C_2}$, $\overrightarrow{C_4}+\overrightarrow{C_5}$, and $\overrightarrow{C_3}+(\overrightarrow{C_4}+\overrightarrow{C_5})$ are all directed cycles.)
\end{lemma}

\begin{proof}
Without loss of generality, we may assume that any two vertices in $G$ are adjacent, because adding an edge to $G$ does not affect the result.  

We use induction on $m$. When $m\leq 2$, the statement is trivial. Assume the statement holds for some integer $m\geq 2$, and we need to show it holds for $m+1$.

Denote $\overrightarrow{C}$ by $(\overrightarrow{e_1},\overrightarrow{P_1},\overrightarrow{e_2},\overrightarrow{P_2}, \cdots,\overrightarrow{e_{m+1}}, \overrightarrow{P_{m+1}})$, where $\overrightarrow{P_i}\subseteq\overrightarrow{C}$ is the directed (internal) path connecting $\overrightarrow{e_i}$ and $\overrightarrow{e_{i+1}}$. See Figure~\ref{Gleb-fig}.

\begin{figure}[ht]
            \centering
            \includegraphics[scale=1]{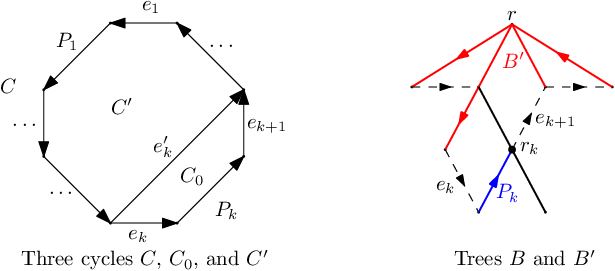}
            \caption{The pictures used in Lemma~\ref{Gleb}. In the right-hand side picture, the tree $B$ is in solid lines and the tree $B'$ is in red solid lines. The dashed lines are some external edges. The directed edges form the directed cycle $\protect\overrightarrow{C}$.}
            \label{Gleb-fig}
\end{figure}

We denote the vertices in an object by $V(\text{object})$. The set $V(\overrightarrow{P_i})$ includes the two endpoints of the path. When $\overrightarrow{P_i}$ contains no arc, we define $V(\overrightarrow{P_i})$ to be the head of $\overrightarrow{e_i}$, which is also the tail of $\overrightarrow{e_{i+1}}$.

We take a vertex $r$ of $G$ viewed as the root of the tree $B$. Define the \emph{height} of a vertex $v$ to be the number of edges in the unique path in $B$ connecting $v$ and $r$. For a (internal) path $\overrightarrow{P_i}$, there exists a unique vertex in $\overrightarrow{P_i}$ with the \emph{minimum} height. We denote the vertex by $r_i$ and define the \emph{height} of $\overrightarrow{P_i}$ to be the height of $r_i$. Let $\overrightarrow{P_k}$ be a path having the \emph{maximal} height among all $\overrightarrow{P_i}$. We remove the vertex $r_k(\neq r)$ together with the incident edges from the tree $B$ and denote the connected component containing $r$ by $B'$. Then $B'$ is a tree not containing any vertex in $\overrightarrow{P_k}$ but containing the vertices in $V(C)\backslash V(\overrightarrow{P_k})$. We will see the definition of $\overrightarrow{P_k}$ and $B'$ is crucial to our proof. 

Without loss of generality, we may assume $1<k<m+1$. Let $\overrightarrow{e_k}'$ be the arc directed from the tail of $\overrightarrow{e_k}$ to the head of $\overrightarrow{e_{k+1}}$. Denote by $\overrightarrow{C_0}$ the directed cycle $(\overrightarrow{e_k},\overrightarrow{P_k},\overrightarrow{e_{k+1}},-\overrightarrow{e_k}')$. Let $\overrightarrow{C}'=\overrightarrow{C}-\overrightarrow{C_0}$. Note that $\overrightarrow{C}'$ is the directed cycle obtained from $\overrightarrow{C}$ by replacing the path $(\overrightarrow{e_k},\overrightarrow{P_k},\overrightarrow{e_{k+1}})$ with the arc $\overrightarrow{e_k}'$. By Lemma~\ref{fundamental}, we have 

\[\overrightarrow{C}'=\sum_{i=1}^{k-1}\overrightarrow{C}(B,\overrightarrow{e_i})+\overrightarrow{C}(B,\overrightarrow{e_k}')+\sum_{i=k+2}^{m+1}\overrightarrow{C}(B,\overrightarrow{e_i}).\] 

Now we apply the induction hypothesis to $\overrightarrow{C}'$ and get a way to completely parenthesize the summation so that the parenthesization has the desired property for $\overrightarrow{C}'$. 

We rewrite $\overrightarrow{C}$ as

\begin{align*}
\overrightarrow{C}=\overrightarrow{C}'+\overrightarrow{C_0} & =  \sum_{i=1}^{k-1}\overrightarrow{C}(B,\overrightarrow{e_i})+(\overrightarrow{C}(B,\overrightarrow{e_k}')+\overrightarrow{C_0})+\sum_{i=k+2}^{m+1}\overrightarrow{C}(B,\overrightarrow{e_i}) \\
 & =  \sum_{i=1}^{k-1}\overrightarrow{C}(B,\overrightarrow{e_i})+(\overrightarrow{C}(B,\overrightarrow{e_k})+\overrightarrow{C}(B,\overrightarrow{e_{k+1}}))+\sum_{i=k+2}^{m+1}\overrightarrow{C}(B,\overrightarrow{e_i}).
\end{align*}

We completely parenthesize the summation for $\overrightarrow{C}$ in the same way as we just did for $\overrightarrow{C}'$ by adding up $(\overrightarrow{C}(B,\overrightarrow{e_k})+\overrightarrow{C}(B,\overrightarrow{e_{k+1}}))$ first and then treating it as the summand $\overrightarrow{C}(B,\overrightarrow{e_k}')$ in $\overrightarrow{C}'$. 

We claim this gives us the desired parenthesization. Indeed, for any new directed cycle $\overrightarrow{D}$ produced in the summation of $\overrightarrow{C}$, there are two cases. \begin{itemize}
    \item If $\overrightarrow{D}$ does not use $\overrightarrow{e_k}$ (and hence $\overrightarrow{e_{k+1}}$), then $\overrightarrow{D}$ also appear in the summation of $\overrightarrow{C}'$. Thus $\overrightarrow{D}$ is a directed cycle. 
    \item If $\overrightarrow{D}$ uses $\overrightarrow{e_k}$, then the corresponding term in the summation of $\overrightarrow{C}'$ is $\overrightarrow{D}'=\overrightarrow{D}-\overrightarrow{C_0}$, where $\overrightarrow{D}'$ could be $\overrightarrow{C}(B,\overrightarrow{e_k}')$ or a newly produced directed cycle containing $\overrightarrow{e_k}'$. Note that the endpoints of all the external edges in $\overrightarrow{C}'$ are in $B'$. So all the fundamental cycles in the summation of $\overrightarrow{C}'$ only use vertices in $B'$, and hence $\overrightarrow{D}'$ does not use any vertex in $\overrightarrow{P_k}$. Thus $\overrightarrow{D}=\overrightarrow{D}'+\overrightarrow{C_0}$ is a directed cycle. 
\end{itemize}
\end{proof}

Now we present an example to show that a circuit signature being acyclic is stronger than being triangulating. (We used a computer program to find the example.)
\begin{proposition}\label{nonex}
There exists a planar graph that admits a triangulating but not acyclic cycle signature. 
\end{proposition}
\begin{proof}
We remove one edge in the complete graph on $5$ vertices and denote the new graph by $G$.

\begin{figure}[ht]
            \centering
            \includegraphics[scale=0.8]{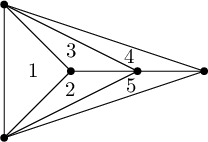}
            \caption{The graph $G$ used in Proposition~\ref{nonex}}
            \label{nonfig}
\end{figure}

The graph $G$ is planar, which allows us to present its directed cycles using regions. We denote by $C_i$ the cycle that bounds the region labeled by $i$ in Figure~\ref{nonfig}, where $i=1,2,3,4,5$. By orienting them counterclockwise, we obtain five directed cycles $\overrightarrow{C_1},\cdots, \overrightarrow{C_5}$. 

Let the cycle signature $\sigma$ be the set of the following directed cycles. The counterclockwise ones are 
$2$, $3$, $5$, $23$, $25$, $123$, $235$, $245$, $345$, $1235$, $2345$, and the clockwise ones are $-1$, $-4$, $-12$, $-13$, $-34$, $-45$, $-125$, $-134$, $-234$, $-1234$, $-12345$, where ``$23$'' means $\overrightarrow{C_2}+\overrightarrow{C_3}$, ``$-234$'' means $-\overrightarrow{C_2}-\overrightarrow{C_3}-\overrightarrow{C_4}$, etc. There are twenty-two cycles in all. 

The signature $\sigma$ is not acyclic because the sum of the directed cycles $123$, $245$, $-234$, and $-125$ is zero. 

It is straightforward to check $\sigma$ is triangulating by Theorem~\ref{tri-sign}(2). (This should be done in minutes by hand. We remark that it is much harder to check $\sigma$ or $\mathcal{A}_\sigma$ is triangulating by definition since there are $75$ spanning trees. )
\end{proof}

\subsection{The Bernardi bijection}\label{Bernardi}
We will apply our theory to recover and generalize some features of the Bernardi bijection in this subsection. For the definition of the Bernardi bijection $f_{\mathcal{A}_\text{B},\mathcal{A}_q^*}$, see Example~\ref{B atlas}.

Note that the internal atlas $\mathcal{A}_q^*$ is a special case of $\mathcal{A^*_{\sigma^*}}$, where $\sigma^*$ is an acyclic cocycle signature \cite[Example 5.1.5]{Yuen2}, so $\mathcal{A}_q^*$ is triangulating. The external atlas $\mathcal{A}_\text{B}$ is not triangulating in general (Remark~\ref{rem-nonexample}). However, it is always dissecting. This fact was discovered and proved by K\'alm\'an and T\'othm\'er\'esz \cite{KT1,KT2} in a different language; see Section~\ref{intro-motivation}. For readers' convenience, we give a proof here.

\begin{lemma}\cite[Lemma 7.7]{KT1}\cite[Proposition 5.6]{KT2}\label{bernardidissect}
The external atlas $\mathcal{A}_\text{B}$ is dissecting. 
\end{lemma}
\begin{proof}
By definition, we need to check  $\overrightarrow{F}=\overrightarrow{B_1}\cap(-\overrightarrow{B_2})$ has a potential cocircuit, where $B_1$ and $B_2$ are two different spanning trees.

Consider the first edge $e_0$ where the Bernardi tours for $B_1$ and $B_2$ differ. Without loss of generality, we may assume $e_0\in B_1$ and $e_0\notin B_2$. Consider the fundamental cocircuit $C^*$ of $e_0$ with respect to $B_1$. We orient it away from $q$ and get the signed cocircuit $\overrightarrow{C^*}$. We will prove that $\overrightarrow{C^*}$ is a potential cocircuit of $\overrightarrow{F}$. See Figure~\ref{lemmapic}.  
\begin{figure}[ht]
            \centering
            \includegraphics[scale=1]{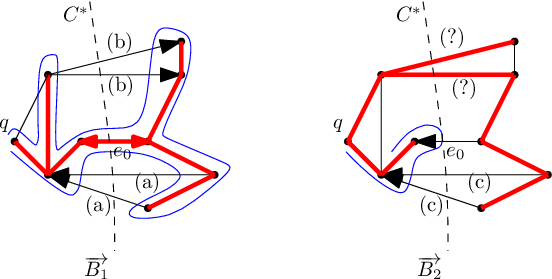}
            \caption{The figure used in the proof of Lemma~\ref{bernardidissect}. The trees are in red. The blue curves are the Bernardi tours. The labels (a), (b), and (c) of edges correspond to the ones in the proof. In $\protect\overrightarrow{B_2}$, the Bernardi tour after its first visit to $e_0$ is arbitrary and hence omitted. The edges in $C^*$ labeled by (?) have almost arbitrary patterns, although in this example they are in the tree. }
            \label{lemmapic}
\end{figure}

Note that the Bernardi tour for $B_1$ uses $e_0$ twice. When it visits $e_0$ the second time, every external edges $f$ in $C^*$ has been cut at least once. Recall that the notation $\overrightarrow{B_1}|_f$ means (the set of) the arc $\overrightarrow{f}$ induced by the tour. There are two cases.

(a) If the tour cuts $f$ before its first visit to $e_0$, then $-\overrightarrow{B_1}|_f\subseteq \overrightarrow{C^*}$.

(b) If the tour cuts $f$ after its first visit to $e_0$, then $\overrightarrow{B_1}|_f\subseteq \overrightarrow{C^*}$. Hence  $\overrightarrow{F}|_f\subseteq \overrightarrow{C^*}$.

Now we look at the Bernardi tour for $B_2$. We know the following two cases.

(c) For every edge $f$ in the case (a), we have $\overrightarrow{B_2}|_f=\overrightarrow{B_1}|_f$ because the two tours coincide until they reach $e_0$. Hence  $\overrightarrow{F}|_f=\emptyset\subseteq \overrightarrow{C^*}$.

(d) For the edge $e_0$, which is external with respect to $B_2$, we have $-\overrightarrow{B_2}|_{e_0}\subseteq \overrightarrow{C^*}$, and hence $\overrightarrow{F}|_{e_0}\subseteq \overrightarrow{C^*}$. 

By (b), (c), and (d), the signed circuit $\overrightarrow{C^*}$ is a potential cocircuit of $\overrightarrow{F}$. 
\end{proof}

Now we may apply Theorem~\ref{main1} and Theorem~\ref{main2} to $f_{\mathcal{A}_\text{B},\mathcal{A}_q^*}$ and get the following results, where Corollary~\ref{Bernardi-extend}(1) recovers the bijectivity of $\overline{f}_{\mathcal{A}_\text{B},\mathcal{A}_q^*}$ proved in \cite{Bernardi}, (2) extends it, and (3) generalizes it. 

\begin{corollary}\label{Bernardi-extend} Let $G$ be a connected ribbon graph.  
\begin{enumerate}
    \item The Bernardi map $f_{\mathcal{A}_\text{B},\mathcal{A}_q^*}$ induces a bijection $\overline{f}_{\mathcal{A}_\text{B},\mathcal{A}_q^*}:B\mapsto [\beta_{(q,e)}]$ between the spanning trees of $G$ and the cycle-cocycle reversal classes of $G$. 
    \item The Bernardi map $f_{\mathcal{A}_\text{B},\mathcal{A}_q^*}$ can be extended to a bijection $\varphi_{\mathcal{A}_\text{B},\mathcal{A}_q^*}$ between subgraphs and orientations in the sense of Theorem~\ref{main2}.
    \item Let $\sigma^*$ be any triangulating cocycle signature. Then the modified Bernardi map $f_{\mathcal{A}_\text{B},\mathcal{A}_{\sigma^*}^*}$ still has the properties (1) and (2). 
\end{enumerate}
\end{corollary}

\section{Lawrence polytopes}\label{Lawrence}
In this section, we will prove Theorem~\ref{3-fold} and Theorem~\ref{main3} introduced in Section~\ref{Lawrence-intro} together with some basic properties of Lawrence polytopes.

For the definitions, see Section~\ref{Lawrence-intro}. Here we recall that $M_{r\times n}$ is a totally unimodular matrix representing a \emph{loopless} regular matroid $\mathcal{M}$. The Lawrence matrix is 
\[\begin{pmatrix} M_{r\times n} & {\bf 0} \\  I_{n\times n} &  I_{n\times n} \end{pmatrix},\]
whose columns are denoted by $P_1, \cdots, P_n, P_{-1}, \cdots, P_{-n}$ in order. The \emph{Lawrence polytope} $\mathcal{P}\in\mathbb{R}^{n+r}$ is the convex hull of the points $P_1, \cdots, P_n, P_{-1}, \cdots, P_{-n}$.  

Due to duality, we will only prove Theorem~\ref{3-fold} and Theorem~\ref{main3} for $\mathcal{P}$. The proof is long, so we divide the section into three parts. 

\subsection{A single maximal simplex of the Lawrence polytope}

The target of this subsection is to characterize maximal simplices of the Lawrence polytope $\mathcal{P}$. 

We start with three basic lemmas, and the proofs are omitted. We denote by \[(x_1, \cdots, x_r, y_1, \cdots, y_n)\] the coordinates of the Euclidean space $\mathbb{R}^{n+r}$ containing $\mathcal{P}$. 
\begin{lemma}\label{affinesubspace}
The Lawrence polytope $\mathcal{P}$ is in the affine subspace $\sum\limits_{i=1}^{n}y_{i}=1$, and the affine subspace does not contain the origin and is of dimension $n+r-1$. 
\end{lemma}

\begin{lemma}\label{affinesimplex}
The convex hull of $k+1$ points $Q_1, \cdots, Q_{k+1}$ in an affine subspace that does not pass through the origin is a $k$-dimensional simplex if and only if the corresponding vectors $Q_1, \cdots, Q_{k+1}$ are linearly independent. 
\end{lemma}

\begin{lemma}\label{linear}
The linear combination $\sum\limits_{i=1}^{n}(a_iP_i+a_{-i}P_{-i})$ is zero if and only if $a_i=-a_{-i}$ for all $i$ and $\sum\limits_{i=1}^{n} a_iM_i=0$, where $M_i$ is the $i$-th column of $M$.
\end{lemma}

\begin{proposition}\label{vertex}
The vertices of $\mathcal{P}$ are the points $P_1, \cdots, P_n, P_{-1}, \cdots, P_{-n}$. 
\end{proposition}

\begin{proof}
It suffices to show that any point $P_i$, where $i$ could be positive or negative, cannot be expressed as a convex combination of the other points. Assume by contradiction that we can do so for some $P_i$. Then by Lemma~\ref{linear}, $P_{-i}$ must have coefficient one in the convex combination, and hence $P_i=P_{-i}$. This contradicts our assumption that $\mathcal{M}$ is loopless.
\end{proof}

Recall that the arcs of $\mathcal{M}$ are denoted by $\overrightarrow{e_1}, \cdots, \overrightarrow{e_n}$ and $\overrightarrow{e_{-1}}, \cdots, \overrightarrow{e_{-n}}$. We denote the underlying edge of the arc $\overrightarrow{e_{i}}$ by $e_i$, where $i>0$. We define the bijection 
    \begin{align*}
    \chi:\{\text{vertices of }\mathcal{P}\} & \to \{\text{arcs of }\mathcal{M}\} \\
    P_i & \mapsto \overrightarrow{e_i}
    \end{align*}

We need the following lemma to characterize the maximal simplices of $\mathcal{P}$.

\begin{lemma}\label{linear2} Let $I\subseteq\{1, \cdots, n, -1, \cdots, -n\}$. Then the set of vectors $\{P_i: i\in I\}$ is linearly dependent if and only if there exists a bioriented circuit in $\{\overrightarrow{e_i}: i\in I\}$, where a bioriented circuit is the union of two opposite signed circuits (as sets of arcs). 
\end{lemma}
\begin{proof}
This is due to Lemma~\ref{linear} and the fact that a collection of columns $M_i$ of $M$ is linearly dependent if and only if the corresponding edges $e_i$ contain a circuit. 
\end{proof}

\begin{corollary}[Theorem~\ref{3-fold}(1)(2)]\label{dim}
\quad

\begin{enumerate}
\item The Lawrence polytope $\mathcal{P}\subseteq\mathbb{R}^{n+r}$ is an $(n+r-1)$-dimensional polytope whose vertices are exactly the points $P_1, \cdots, P_n, P_{-1}, \cdots, P_{-n}$. 

\item The map $\chi$ induces a bijection (still denoted by $\chi$)
    \begin{align*}
    \chi:\{\text{maximal simplices of }\mathcal{P}\} & \to \{\text{externally oriented bases of }\mathcal{M}\} \\
    \begin{gathered}
        \text{a maximal simplex}\\
        \text{with vertices }\{P_i:i\in I\}
    \end{gathered} & \mapsto \text{the fourientation }\{\chi(P_i):i\in I\}. 
    \end{align*}
\end{enumerate}

\end{corollary}

\begin{proof}
We have proved the part on the vertices of $\mathcal{P}$ (Proposition~\ref{vertex}). 

Then we prove (2). Let $\overrightarrow{F}$ be a set of arcs of $\mathcal{M}$. By Lemma~\ref{linear2}, 
Lemma~\ref{affinesubspace}, and Lemma~\ref{affinesimplex}, the set $\overrightarrow{F}$ does not contain a bioriented circuit if and only if the corresponding vertices of $\mathcal{P}$ (via $\chi^{-1}$) form a simplex. When the set $\overrightarrow{F}$ does not contain a bioriented circuit, its bioriented edges are independent in $\mathcal{M}$. Hence the cardinality of $\overrightarrow{F}$ satisfies $|\overrightarrow{F}|\leq n+r$, and the equality holds if and only if $\overrightarrow{F}$ is an externally oriented basis of $\mathcal{M}$. Thus (2) holds. 

Now we complete the proof of (1). Because the dimension of a maximal simplex is $n+r-1$, the dimension of $\mathcal{P}$ is also $n+r-1$ by Lemma~\ref{affinesubspace}. 
\end{proof}

\subsection{Two maximal simplices of the Lawrence polytope}
To show Theorem~\ref{3-fold}(3), which characterizes the triangulations and dissections of $\mathcal{P}$, we first prove Proposition~\ref{local}, which characterizes when two maximal simplices satisfy (II) and (III) in Definition~\ref{tri-diss-def}, respectively. 

Note that when we say two simplices or two fourientations, they might be identical. 

We need some preparations.
\begin{definition}
(1) If the vertices of a simplex $S$ are some of the vertices of $\mathcal{P}$, then $S$ is called a \emph{simplex of $\mathcal{P}$}.

(2) The relative interior of $S$ is denoted by $S^\circ$. 
\end{definition}
The following lemma is basic, and the proof is omitted. 

\begin{lemma}\label{unique}
Let $S$ be a simplex and $x\in S$. Then the point $x$ can be uniquely written as a convex combination of the vertices of $S$. Moreover, 
$x\in S^\circ$ if and only if each vertex of $S$ has a nonzero coefficient in the convex combination. 
\end{lemma}

The following lemma gives an equivalent description of (III) in Definition~\ref{tri-diss-def}. The book \cite[Definition 2.3.1]{DRS} uses it as a definition. The proof is omitted. 
\begin{lemma}\label{tri-lemma}
Let $S_1$ and $S_2$ be two maximal simplices of $\mathcal{P}$. Then $S_1$ and $S_2$ intersect in a common face if and only if for any face $A_1$ of $S_1$ and any face $A_2$ of $S_2$ such that $A_1^\circ\cap A_2^\circ\neq\emptyset$, we have $A_1=A_2$. 
\end{lemma}

We aim at proving Lemma~\ref{hard lemma}, which characterizes $A_1^\circ\cap A_2^\circ=\emptyset$ in terms of fourientations. Before proving it, we prove a lemma characterizing when a fourientation is ``totally cyclic''. 

\begin{lemma}\label{lemma for revision}
Let $\overrightarrow{F}$ be a fourientation of $\mathcal{M}$. Then every one-way oriented edge in $\overrightarrow{F}$ belongs to a potential circuit of $\overrightarrow{F}$ if and only if there exists a vector $(u_i)\in\ker_\mathbb{R}(M)$ such that 
\begin{itemize}
    \item $\overrightarrow{F}|_{e_i}=\emptyset$ implies $u_i=0$,
    \item $\overrightarrow{F}|_{e_i}=\{\overrightarrow{e_i}\}$ implies $u_i>0$,
    \item and $\overrightarrow{F}|_{e_i}=\{\overrightarrow{e_{-i}}\}$ implies $u_i<0$.
    \end{itemize}

\end{lemma}

\begin{proof}
First we prove the ``only if'' part. For every one-way oriented edge $e$ in $\overrightarrow{F}$, let $\overrightarrow{C_e}$ be a potential circuit of $\overrightarrow{F}$ containing it. Viewing $\overrightarrow{C_e}$ as a vector in $\mathbb{R}^n$, we have $\overrightarrow{C_e}\in\ker_\mathbb{R}(M)$. Let \[\overrightarrow{u}=\sum \overrightarrow{C_e},\] where the sum is over all the one-way oriented edges $e$ in $\overrightarrow{F}$. If there is no one-way oriented edge in $\overrightarrow{F}$, the vector $\overrightarrow{u}$ is zero. Then $\overrightarrow{u}\in\ker_\mathbb{R}(M)$ satisfies the three desired properties. 

For the ``if'' part, by Lemma~\ref{conformal}, we may decompose the vector $(u_i)$ into a linear combination $\sum k_j\overrightarrow{C_j}$ of signed circuits, where the coefficients $k_j$ are positive and for each edge $e_i$ of each $C_j$, the sign of $e_i$ in $\overrightarrow{C_j}$ agrees with the sign of $u_i$. Hence every one-way oriented edge in $\overrightarrow{F}$ belongs to at least one $\overrightarrow{C_j}$ in the sum, and every $\overrightarrow{C_j}$ is a potential circuit of $\overrightarrow{F}$.

\end{proof}

To state Lemma~\ref{hard lemma}, we introduce a notation. For a fourientation $\overrightarrow{F}$, we denote its underlying edges by
\[E(\overrightarrow{F})=\{e\in E: \overrightarrow{F}|_e\neq\emptyset\}.\]

\begin{lemma}\label{hard lemma}
Assume $S_1$ and $S_2$ are two simplices of $\mathcal{P}$ (not necessarily maximal). Let $\overrightarrow{F_k}$ be the fourientation $\chi(S_k)$ for $k=1,2$, and denote $\overrightarrow{F}=\overrightarrow{F_1}\cap(-\overrightarrow{F_2})$. Then $S_1^\circ\cap S_2^\circ\neq\emptyset$ if and only if $E(\overrightarrow{F_1})=E(\overrightarrow{F_2})$ and every one-way oriented edge in $\overrightarrow{F}$ belongs to a potential circuit of $\overrightarrow{F}$.
\end{lemma}

\begin{proof}
We first prove the ``only if'' part. Let $x\in S_1^\circ\cap S_2^\circ$. Throughout the proof, $k\in\{1,2\}$. We denote by $F_k$ the set of \emph{indices} of the edges in $E(\overrightarrow{F_k})$ ($F_k\subseteq\{1,\cdots, n\})$.   

By Lemma~\ref{unique}, we may write 
\begin{equation}\label{eq in section 4}
x=\sum_{i\in F_1}(w_{1i}^+P_i+w_{1i}^-P_{-i})=\sum_{i\in F_2}(w_{2i}^+P_i+w_{2i}^-P_{-i}),    
\end{equation}
where the nonnegative coefficients $w_{ki}^+, w_{ki}^-$ sum up to $1$ for each $k$. Note that only when the edge $e_i$ is one-way oriented in $\overrightarrow{F_k}$, one of $w_{ki}^+$ and  $w_{ki}^-$ is zero, and all the other coefficients are positive. 

We compare the two convex combinations of $x$ using Lemma~\ref{linear}. From the lower half of the Lawrence matrix, we get $F_1=F_2$ and $w_{1i}^++ w_{1i}^-=w_{2i}^++w_{2i}^-$ for $i\in F_1$. Denote \[w_i=w_{1i}^++ w_{1i}^-.\] It is clear that $\sum\limits_{i\in F_1}w_i=1$ and each summand $w_i>0$. 

Now we focus on the upper half of the Lawrence matrix. By comparing the two convex combinations of $x$ restricted to the top $r$ entries, we obtain Table~\ref{table3}, which we explain below. For every $i\in F_1$, according to the status of the edge $e_i$ in $\overrightarrow{F}$, there are $4$ possible types given in the first column of Table~\ref{table3}. We omit ``$|_{e_i}$'' after $\overrightarrow{F}$ and $\overrightarrow{F_k}$ (e.g.$\overrightarrow{F}=\updownarrow$ means $\overrightarrow{F}|_{e_i}=\updownarrow$). These $4$ types can be further divided into $9$ types according to how $\overrightarrow{F_1}$ and $\overrightarrow{F_2}$ orient $e_i$, and neither $\overrightarrow{F_1}$ nor $\overrightarrow{F_2}$ could be empty over $e_i$ because $F_1=F_2$. Then for each of the $9$ types,  we know which of $w_{1i}^+, w_{1i}^-, w_{2i}^+$, and $w_{2i}^-$ are zero. For example, when $e_i$ is of the 4th type, we have $P_i\in S_1$ and $P_{-i}\notin S_1$ because $\overrightarrow{F_k}=\chi(S_k)$. Hence $w_{1i}^+P_i=w_iP_i$ and $w_{1i}^-P_{-i}=0$. We compute the top $r$ entries of $w_{1i}^+P_i+w_{1i}^-P_{-i}-w_{2i}^+P_i-w_{2i}^-P_{-i}$ as listed in the last column of Table~\ref{table3}. Because the top $r$ entries of $P_i$ and $P_{-i}$ are $M_i$ and $0$ respectively, the vector $w_{1i}^+P_i+w_{1i}^-P_{-i}-w_{2i}^+P_i-w_{2i}^-P_{-i}$ must be of the form $u_iM_i$, where $u_i=w_{1i}^+-w_{2i}^+$.

\begin{table}
\centering
\bgroup
\def\arraystretch{1.5}
\begin{tabular}{ |m{2.2cm}|m{4cm}|c|c|c|c|c| }

\hline
type of $e_i$ in terms of $\overrightarrow{F}$ & type of $e_i$ in terms of $\overrightarrow{F_k}$ and label  & $w_{1i}^+P_i$ & $w_{1i}^-P_{-i}$ & $w_{2i}^+P_i$ & $w_{2i}^-P_{-i}$ & $u_iM_i$\\
\hline

$\overrightarrow{F}=\updownarrow$ &  1: $\overrightarrow{F_1}=\overrightarrow{F_2}=\updownarrow$  & $w_{1i}^+P_i$ & $w_{1i}^-P_{-i}$& $w_{2i}^+P_i$ &  $w_{2i}^-P_{-i}$ & $(w_{1i}^+-w_{2i}^+)M_i$ \\

\hline

\multirow{2}{2.2cm}{$\overrightarrow{F}=\emptyset$}  &2: $\overrightarrow{F_1}=\overrightarrow{F_2}=\overrightarrow{e_i}$ &$w_{i}P_i$ & 0 & $w_{i}P_i$ &  0 & $0$ \\ 
\cline{2-7}

&3: $\overrightarrow{F_1}=\overrightarrow{F_2}=\overrightarrow{e_{-i}}$ &  0 & $w_{i}P_{-i}$ & 0 &  $w_{i}P_{-i}$ & $0$ \\

\hline

\multirow{3}{2.2cm}{$\overrightarrow{F}=\overrightarrow{e_i}$}  & 4: $\overrightarrow{F_1}=\overrightarrow{e_i}$, $\overrightarrow{F_2}=\updownarrow$ & $w_{i}P_i$ & 0 & $w_{2i}^+P_i$ &  $w_{2i}^-P_{-i}$  & $w_{2i}^-M_i$ \\ 
\cline{2-7}

& 5: $\overrightarrow{F_1}=\updownarrow$, $\overrightarrow{F_2}=\overrightarrow{e_{-i}}$ &  $w_{1i}^+P_i$ & $w_{1i}^-P_{-i}$& 0 &  $w_{i}P_{-i}$ & $w_{1i}^+M_i$ \\

\cline{2-7}

& 6: $\overrightarrow{F_1}=\overrightarrow{e_i}$, $\overrightarrow{F_2}=\overrightarrow{e_{-i}}$ &  $w_{i}P_i$ & 0 & 0 &  $w_{i}P_{-i}$ & $w_iM_i$ \\
\hline

\multirow{3}{2.2cm}{$\overrightarrow{F}=\overrightarrow{e_{-i}}$}  & 7: $\overrightarrow{F_1}=\overrightarrow{e_{-i}}$, $\overrightarrow{F_2}=\updownarrow$  & 0 & $w_{i}P_{-i}$ & $w_{2i}^+P_i$ &  $w_{2i}^-P_{-i}$  & $-w_{2i}^+M_i$ \\ 
\cline{2-7}

& 8: $\overrightarrow{F_1}=\updownarrow$, $\overrightarrow{F_2}=\overrightarrow{e_i}$ &  $w_{1i}^+P_i$ & $w_{1i}^-P_{-i}$& $w_{i}P_i$ & 0 & $-w_{1i}^-M_i$ \\

\cline{2-7}

& 9: $\overrightarrow{F_1}=\overrightarrow{e_{-i}}$, $\overrightarrow{F_2}=\overrightarrow{e_i}$ & 0 & $w_{i}P_{-i}$ & $w_{i}P_i$ &  0 & $-w_iM_i$ \\
\hline

\end{tabular}
\egroup
\caption{The computational results used in Lemma~\ref{hard lemma}. Here $i\in F_1=F_2$, $M_i$ is the $i$-th column of $M$, and $u_{i}M_i$ is the top $r$ entries of $w_{1i}^+P_i+w_{1i}^-P_{-i}-w_{2i}^+P_i-w_{2i}^-P_{-i}$ (which must be a scalar multiple of $M_i$). } 
\label{table3}

\vspace{-8mm}
\end{table}

By equation~(\ref{eq in section 4}), we have \[\sum_{i\in F_1}u_iM_i=0.\]
For every $i\notin F_1$, set $u_i=0$. Then we have $(u_i)\in\ker_\mathbb{R}(M)$. By comparing the first column with the last column of Table~\ref{table3}, we observe that, for each edge $e_i$ that is not bioriented in $\overrightarrow{F}$ (types $2$ to $9$), the sign of $u_i$ agrees with the sign of $e_i$ in $\overrightarrow{F}$. Hence we may apply Lemma~\ref{lemma for revision} to $(u_i)$, which completes the proof of the ``only if'' part.  

For the ``if'' part, our proof strategy is to reverse the proof of the ``only if'' part. Note that the second column of Table~\ref{table3} still lists all the possible types of edges in $E(\overrightarrow{F_1})(=E(\overrightarrow{F_2}))$. By Lemma~\ref{lemma for revision}, there exists a vector $(u_i)\in\ker_\mathbb{R}(M)$ such that for every $e_i$ that is not bioriented in $\overrightarrow{F}$, the sign pattern of $e_i$ agrees with the sign pattern of $u_i$. This implies that the vector $(u_i)$ agrees with the sign pattern of the last column of Table~\ref{table3} from type 2 to type 9, although we have not specified the weights ``$w$'' yet.

Then for each $i\in F_1$, we find nonnegative numbers $\Tilde{w}_{1i}^+$, $\Tilde{w}_{1i}^-$, $\Tilde{w}_{2i}^+$, and $\Tilde{w}_{2i}^-$ such that  \begin{itemize}
    \item $\Tilde{w}_{1i}-\Tilde{w}_{2i}=u_i$, 
    \item $\Tilde{w}_{1i}^++ \Tilde{w}_{1i}^-=\Tilde{w}_{2i}^++\Tilde{w}_{2i}^-$,  
    \item and the sign pattern agrees with the second column of Table~\ref{table3} (or the first column of Table~\ref{table4}). In other words, $\Tilde{w}_{ki}^+=0\Leftrightarrow \overrightarrow{e_{i}}\notin\overrightarrow{F_k}|_{e_i}$ and $\Tilde{w}_{ki}^-=0\Leftrightarrow \overrightarrow{e_{-i}}\notin\overrightarrow{F_k}|_{e_i}$. 
\end{itemize}

There are infinitely many choices for these $\Tilde{w}$'s; see Table~\ref{table4} for a choice. Lastly, we normalize the weights $\Tilde{w}_{ki}^+$ and $\Tilde{w}_{ki}^-$ to obtain $w_{ki}^+$ and $w_{ki}^-$ such that the total sum is $1$ for each $k$. Then we have $\sum\limits_{i\in F_1}(w_{1i}^+P_i+w_{1i}^-P_{-i})\in S_1^\circ$, $\sum\limits_{i\in F_2}(w_{2i}^+P_i+w_{2i}^-P_{-i})\in S_2^\circ$, and $\sum\limits_{i\in F_1}(w_{1i}^+P_i+w_{1i}^-P_{-i})=\sum\limits_{i\in F_2}(w_{2i}^+P_i+w_{2i}^-P_{-i})$ as desired.  

\begin{table}
\centering
\bgroup
\def\arraystretch{1.5}
\begin{tabular}{ |m{4cm}|c|c|c|c|c| }

\hline
type of $e_i$ in terms of $\overrightarrow{F_k}$ and label  & $\Tilde{w}_{1i}$ & $\Tilde{w}_{1i}$ & $\Tilde{w}_{2i}^+$ & $\Tilde{w}_{2i}^-$ & sign of $u_i$\\
\hline

1: $\overrightarrow{F_1}=\overrightarrow{F_2}=\updownarrow$  & $|u_i|+1$ & $|u_i|-u_i+1$ & $|u_i|-u_i+1$ &  $|u_i|+1$ & $+,-,0$ \\

\hline

2: $\overrightarrow{F_1}=\overrightarrow{F_2}=\overrightarrow{e_i}$ &$1$ & 0 & $1$ &  0 & $0$ \\ 

\hline

3: $\overrightarrow{F_1}=\overrightarrow{F_2}=\overrightarrow{e_{-i}}$ &  0 & $1$ & 0 &  $1$ & $0$ \\

\hline

4: $\overrightarrow{F_1}=\overrightarrow{e_i}$, $\overrightarrow{F_2}=\updownarrow$ & $2u_{i}$ & 0 & $u_{i}$ &  $u_{i}$  & $+$ \\ 
\hline

5: $\overrightarrow{F_1}=\updownarrow$, $\overrightarrow{F_2}=\overrightarrow{e_{-i}}$ &  $u_{1}$ & $u_{i}$ & 0 &  $2u_{i}$ & $+$ \\

\hline

6: $\overrightarrow{F_1}=\overrightarrow{e_i}$, $\overrightarrow{F_2}=\overrightarrow{e_{-i}}$ &  $u_{i}$ & 0 & 0 &  $u_{i}$ & $+$ \\
\hline

7: $\overrightarrow{F_1}=\overrightarrow{e_{-i}}$, $\overrightarrow{F_2}=\updownarrow$  & 0 & $-2u_i$ & $-u_{i}$ &  $-u_{i}$  & $-$ \\ 
\hline

8: $\overrightarrow{F_1}=\updownarrow$, $\overrightarrow{F_2}=\overrightarrow{e_i}$ &  $-u_{i}$ & $-u_{i}$ & $-2u_{i}$ & 0 & $-$ \\

\hline

9: $\overrightarrow{F_1}=\overrightarrow{e_{-i}}$, $\overrightarrow{F_2}=\overrightarrow{e_i}$ & 0 & $-u_{i}$ & $-u_{i}$ &  0 & $-$ \\
\hline

\end{tabular}
\egroup
\caption{The construction of $\Tilde{w}_{1i}^+$, $\Tilde{w}_{1i}^-$, $\Tilde{w}_{2i}^+$, and $\Tilde{w}_{2i}^-$ used in the proof of Lemma~\ref{hard lemma} } 
\label{table4}

\vspace{-8mm}
\end{table}

\end{proof}

We omit the proof of the following simple lemma.
\begin{lemma}\label{stupid}
Let $\overrightarrow{F_1}$ and $\overrightarrow{F_2}$ be two fourientations of $\mathcal{M}$ such that $E(\overrightarrow{F_1})=E(\overrightarrow{F_2})$. Then $\overrightarrow{F_1}\neq\overrightarrow{F_2}$ if and only if $\overrightarrow{F}=\overrightarrow{F_1}\cap(-\overrightarrow{F_2})$ contains a one-way oriented edge.
\end{lemma}

We are ready to prove the main result of this subsection. 
\begin{proposition}\label{local}
Let $S_1$ and $S_2$ be two maximal simplices of $\mathcal{P}$. Let $\overrightarrow{B_k}$ be the externally oriented basis $\chi(S_k)$ for $k=1,2$, and denote $\overrightarrow{F}=\overrightarrow{B_1}\cap(-\overrightarrow{B_2})$. 
\begin{enumerate}
    \item $S_1^\circ\cap S_2^\circ=\emptyset$ if and only if $\overrightarrow{F}$ has a potential cocircuit. 
    \item $S_1$ and $S_2$ intersect at a common face if and only if  $\overrightarrow{F}$ has no potential circuit.
    \item If $\overrightarrow{F}$ has a potential cocircuit or has no potential circuit, then $S_1\neq S_2$ implies $B_1\neq B_2$. 
\end{enumerate}
\end{proposition}

\begin{proof}
First we prove (3). Assume by contradiction that $B_1=B_2$. Then $\overrightarrow{F}$ is a fourientation where the internal edges are all bioriented and there exists an external edge that is one-way oriented due to  $\overrightarrow{B_1}\neq \overrightarrow{B_2}$. So, $\overrightarrow{F}$ has no potential cocircuit and has a potential circuit, which contradicts the assumption. 

For (1), we apply Lemma~\ref{hard lemma} to $S_1$ and $S_2$. Since $E(\overrightarrow{B_1})=E(\overrightarrow{B_2})$ always holds, $S_1^\circ\cap S_2^\circ=\emptyset$ if and only if there exist a one-way oriented edge in $\overrightarrow{F}$ such that it does not belong to any potential circuit of $\overrightarrow{F}$. By Lemma~\ref{3-painting}, we find a potential cocircuit of $\overrightarrow{F}$.

For (2), we apply Lemma~\ref{tri-lemma}. The maximal simplices $S_1$ and $S_2$ do not intersect in a common face if and only if there exist two distinct faces $A_1$ of $S_1$ and $A_2$ of $S_2$ such that $A_1^\circ\cap A_2^\circ\neq\emptyset$, which by Lemma~\ref{hard lemma} is equivalent to

\begin{center}
($\star$) there exist two distinct fourientations $\overrightarrow{F_1}\subseteq\overrightarrow{B_1}$ and $\overrightarrow{F_2}\subseteq\overrightarrow{B_2}$ such that $E(\overrightarrow{F_1})=E(\overrightarrow{F_2})$ and every one-way oriented edge in $\overrightarrow{F_0}:=\overrightarrow{F_1}\cap(-\overrightarrow{F_2})$ belongs to a potential circuit of $\overrightarrow{F_0}$.
\end{center}

It remains to show ($\star$) is equivalent to $\overrightarrow{F}$ having a potential circuit. If ($\star$) holds, then $\overrightarrow{F_0}\subseteq\overrightarrow{F}$, and hence a potential circuit of  $\overrightarrow{F_0}$ is also a potential circuit of $\overrightarrow{F}$. By Lemma~\ref{stupid}, there is indeed a one-way oriented edge in $\overrightarrow{F_0}$. Thus $\overrightarrow{F}$ has a potential circuit. Conversely, if $\overrightarrow{F}$ has a potential circuit $\overrightarrow{C}$, then there must be a one-way oriented edge in $\overrightarrow{F}|_C$ (because in general $\overrightarrow{B_1}\cap(-\overrightarrow{B_2})$ does not contain bioriented circuits). Set $\overrightarrow{F_1}=\overrightarrow{B_1}|_C$ and $\overrightarrow{F_2}=\overrightarrow{B_2}|_C$. Clearly, we have $E(\overrightarrow{F_1})=E(\overrightarrow{F_2})$ and $\overrightarrow{F_0}=\overrightarrow{F}|_C$. By Lemma~\ref{stupid},  $\overrightarrow{F_1}\neq\overrightarrow{F_2}$. So, ($\star$) holds. 
\end{proof}

\begin{Rem}
Proposition~\ref{local}(2) is analogous to \cite[Lemma 12.6]{P}; see Section~\ref{intro-motivation} for details. 
\end{Rem}
    
\subsection{Volume of the Lawrence polytope and Theorem~\ref{3-fold}(3)}
Proposition~\ref{local} is close to Theorem~\ref{3-fold}(3). It remains to show the maximal simplices coming from a dissecting atlas or a triangulating atlas indeed cover the Lawrence polytope $\mathcal{P}$. Unfortunately, we cannot find a direct proof showing that any point of $\mathcal{P}$ is in some maximal simplex that comes from the given atlas. Instead, we make use of volume. 

Recall that $\mathcal{P}$ is in the affine space $\sum\limits_{i=1}^ny_i=1$, and the affine space is in $\mathbb{R}^{n+r}$ with coordinate system $(x_1, \cdots, x_r, y_1, \cdots, y_n)$.

For a polytope $S$, we denote by $\text{vol}(S)$ the volume of $S$. 

We first compute the volume of a maximal simplex of $\mathcal{P}$. 

\begin{lemma}\label{volume-simplex}
Let $S$ be a maximal simplex of $\mathcal{P}$. Then \[\text{vol}(S)=\frac{\sqrt{n}}{(n+r-1)!},\]
where $n$ is the number of edges and $n+r-1$ is the dimension of $\mathcal{P}$. In particular, all the maximal simplices of $\mathcal{P}$ have the same volume. 
\end{lemma}

\begin{proof}
Let $\Tilde{S}$ be the pyramid whose base is $S$ and whose apex is the origin $O$. 

The height of $\Tilde{S}$ is the distance from $O$ to the affine hyperplane $\sum\limits_{i=1}^ny_i=1$, so \[\text{vol}(\Tilde{S})=\frac{1}{\dim(\Tilde{S})}\cdot\text{base}\cdot\text{height}=\frac{1}{n+r}\cdot \text{vol}(S)\cdot\frac{1}{\sqrt{n}}.\]

Another way to compute $\text{vol}(\Tilde{S})$ is using a determinant. Note that $\Tilde{S}$ is a simplex and one of its vertices is $O$. The coordinates of the $n+r$ vertices of $S$ are the corresponding columns of the Lawrence matrix \[\begin{pmatrix} M_{r\times n} & {\bf 0} \\  I_{n\times n} &  I_{n\times n} \end{pmatrix}.\]
Thus they form a $(n+r)\times(n+r)$ submatrix $N$. Hence
\[\text{vol}(\Tilde{S})=\frac{1}{\dim(\Tilde{S})!}\cdot|\det(N)|.\]

Because $M$ is totally unimodular, appending a standard unit vector to it still results in a totally unimodular matrix. By doing this repeatedly, we see that the Lawrence matrix is totally unimodular. Hence $\det(N)=\pm 1$ and \[\text{vol}(\Tilde{S})=\frac{1}{(n+r)!}.\]
By combining the two formulas of $\text{vol}(\Tilde{S})$, we get the desired formula.

\end{proof}

We denote 
\[t=\text{the number of the maximal simplices used in a dissection of }\mathcal{P}.\]
By Lemma~\ref{volume-simplex}, the number $t$ does not depend on the choice of dissections. It is well known that the \emph{regular triangulation} of a polytope always exists; see \cite[Section 2.2.1]{DRS}. Hence a triangulation (and hence a dissection) of $\mathcal{P}$ always exists. 

We denote 
\[b=\text{the number of bases of }\mathcal{M}.\]

The target is to show $t=b$. The following inequality is implied by the existence of a triangulation. 

\begin{lemma}\label{t<b}
$t\leq b$.    
\end{lemma}
\begin{proof}
Because there exists a triangulation of $\mathcal{P}$, by Proposition~\ref{local}(2)(3), the externally oriented bases corresponding to the maximal simplices in the triangulation have \emph{distinct} underlying bases. Thus $t\leq b$.
\end{proof}

The other direction is implied by the existence of a triangulating atlas. To show the existence of a triangulating atlas, we show the existence of an acyclic signature, which is implicitly proved in \cite{BBY} by making use of the following equivalent definition of acyclic signatures. 

\begin{lemma}\cite[Lemma 3.1.1]{BBY}\label{geometric signature}
Let $\sigma$ be a circuit signature of $\mathcal{M}$. Then $\sigma$ is acyclic if and only if there exists $\overrightarrow{w}\in\mathbb{R}^E$ such that $\overrightarrow{w}\cdot \overrightarrow{C}>0$ for each signed circuit $\overrightarrow{C}\in\sigma$, where the product is the usual inner product. 
\end{lemma}

\begin{lemma}\label{exist}
Let $\mathcal{M}$ be a regular matroid. 
\begin{enumerate}
    \item \cite{BBY} There exists an acyclic circuit signature if $\mathcal{M}$ has at least one circuit. 
    \item There exists a triangulating external atlas.
\end{enumerate}
\end{lemma}

\begin{proof}
(1) We can always find $\overrightarrow{w}\in\mathbb{R}^E$ such that $\overrightarrow{w}\cdot \overrightarrow{C}\neq 0$ for any signed circuit $\overrightarrow{C}$. Then let $\sigma$ be the collection of signed circuits $\overrightarrow{C}$ with the property $\overrightarrow{w}\cdot \overrightarrow{C}>0$. By Lemma~\ref{geometric signature}, $\sigma$ is acyclic. 

(2) If $\mathcal{M}$ has at least one circuit, then by Lemma~\ref{acyclic-tri}, $\mathcal{A}_\sigma$ is triangulating. If $\mathcal{M}$ has no circuit, then the ground set $E$ is the unique basis of $\mathcal{M}$. In this case, it is easy to see that $\mathcal{M}$ has a unique external atlas $\mathcal{A}$, which is triangulating. 
\end{proof}

\begin{lemma}\label{t>b}
$t\geq b$.    
\end{lemma}
\begin{proof}
Because there exists a triangulating external atlas, by Proposition~\ref{local}(2), the polytope $\mathcal{P}$ contains at least $b$ distinct maximal simplices any two of which intersect at a common face. Thus $t\geq b$. 
\end{proof}

\begin{corollary}
The volume of the Lawrence polytope $\mathcal{P}$ is \[\text{vol}(\mathcal{P})=(\text{the number of bases of }\mathcal{M})\cdot\frac{\sqrt{n}}{(n+r-1)!}.\]
\end{corollary}
\begin{proof}
This is a direct consequence of Lemma~\ref{volume-simplex}, Lemma~\ref{t<b}, and Lemma~\ref{t>b}. 
\end{proof}

Now we are ready to prove Theorem~\ref{3-fold}(3). 

\begin{theorem}[Theorem~\ref{3-fold}(3)]
The map $\chi$ induces two bijections
    \begin{align*}
    \chi:\{\text{triangulations of }\mathcal{P}\} & \to \{\text{triangulating external atlases of }\mathcal{M}\} \\
    \begin{gathered}
        \text{a triangulation with} \\
        \text{maximal simplices }\{S_i:i\in I\}
    \end{gathered} & \mapsto \text{the external atlas }\{\chi(S_i):i\in I\}, 
    \end{align*}
    and \begin{align*}
    \chi:\{\text{dissections of }\mathcal{P}\} & \to \{\text{dissecting external atlases of }\mathcal{M}\} \\
     \begin{gathered}
         \text{a dissection with}\\\text{maximal simplices }\{S_i:i\in I\}
     \end{gathered} & \mapsto \text{the external atlas }\{\chi(S_i):i\in I\}.
    \end{align*}
\end{theorem}

\begin{proof}
This is a direct consequence of Proposition~\ref{local} and the fact $t=b$. 
\end{proof}

\subsection{Regular triangulations and acyclic signatures}\label{regular-sign}
For the basics of regular triangulations, we refer the readers to \cite{LS} and \cite[Chapter 2]{DRS}. Here we recall the construction of regular triangulations of a polytope $\mathcal{P}\subseteq\mathbb{R}^n$ with the vertex set $V$. 
\begin{enumerate}
    \item[(i)] Pick a \emph{height function} $h:V\rightarrow \mathbb{R}$. \emph{Lift} each vertex $v\in V$ in $\mathbb{R}^n$ to $\mathbb{R}^{n+1}$ by appending $h(v)$ to $v$ as the $(n+1)$-th coordinate. Take the convex hull of the lifted vertices and get a lifted polytope $\mathcal{P}'$. 
    \item[(ii)] Project the \emph{lower facets} of $\mathcal{P}'$ onto $\mathbb{R}^n$. Here, a lower facet is a facet that is visible from below (i.e., a facet whose outer normal vector has its last coordinate negative). 
    \item[(iii)] When all the lower facets are simplices, the projected facets form a triangulation of $\mathcal{P}$, called a \emph{regular triangulation}. (See \cite{DRS} for a proof.)    
\end{enumerate}

In the next few paragraphs, we fix the notations for this subsection. 

Let $\mathcal{P}$ be the Lawrence polytope of a loopless regular matroid $\mathcal{M}$. We use bold letters to represent column vectors. In particular, the vertices $P_i$ of $\mathcal{P}$ will be denoted by $\mathbf{P}_i$ instead. Let $(x_1, \cdots, x_r, y_1, \cdots, y_n)^T$ denote the coordinate of a point in $\mathbb{R}^{n+r}$. 

Recall that $\mathcal{P}$ spans the affine hyperplane $\sum\limits_{i=1}^{n}y_{i}=1$ (Lemma~\ref{affinesubspace} and Corollary~\ref{dim}). To lift the vertices $\mathbf{P}_i$ of $\mathcal{P}$, we use the space $\mathbb{R}^{n+r}$ and lift $\mathbf{P}_i$ along the normal vector \[\mathbf{n}_P=(0, \cdots, 0, 1, \cdots, 1)^T\] of the affine hyperplane spanned by $\mathcal{P}$. To be precise, for every $i=1, \cdots, n, -1, \cdots, -n$, we lift $\mathbf{P}_i$ to the point \[\mathbf{P}_i':=\mathbf{P}_i+h_i\cdot\mathbf{n}_P,\] where $h_i\in\mathbb{R}$. The lifted polytope $\mathcal{P}'$ is the convex hull of these lifted points $\mathbf{P}_i'$.

Fix an arbitrary maximal simplex $S$ of $\mathcal{P}$. Let $S'$ be the convex hull of the lifted points $\{\mathbf{P}_i':\mathbf{P}_i\in S\}$. 

\begin{lemma}
The polytope $S'$ is a simplex of dimension $n+r-1$. 
\end{lemma}
\begin{proof}
Assume for the sake of contradiction that the set $\{\mathbf{P}'_i:\mathbf{P}_i\in S\}$ is affinely dependent. Then there exists a vertex $\mathbf{P}_{i_0}\in S$ such that\[\sum_{\mathbf{P}_i\in S,i\neq i_0}k_i(\mathbf{P}'_i-\mathbf{P}'_{i_0})=0,\]where the real coefficients $k_i$ are not all zeros. This implies\[\sum_{\mathbf{P}_i\in S,i\neq i_0}k_i(\mathbf{P}_i-\mathbf{P}_{i_0})=-\sum_{\mathbf{P}_i\in S,i\neq i_0}k_i(h_i-h_{i_0})\mathbf{n}_P.\]Because the normal vector $\mathbf{n}_P$ is orthogonal to every $\mathbf{P}_i-\mathbf{P}_{i_0}$, we have \[\sum_{\mathbf{P}_i\in S,i\neq i_0}k_i(\mathbf{P}_i-\mathbf{P}_{i_0})=0,\]which contradicts that the set $\{\mathbf{P}_i:\mathbf{P}_i\in S\}$ is affinely independent.  
\end{proof}

We need to use the (arbitrary) data $h_i$'s and $\mathcal{P}$ to determine whether $S'$ is a lower facet of $\mathcal{P}'$. Let $H$ be the unique affine hyperplane of $\mathbb{R}^{n+r}$ that contains $S'$. Then $S'$ is a lower facet of $\mathcal{P}'$ if and only if for any vertex $\mathbf{P}_j$ of $\mathcal{P}$ not in $S$, the corresponding lifted point $\mathbf{P}_j'$ is higher than $H$ (with respect to the direction $\mathbf{n}_P$). 

The following lemma shows that $H$ is not parallel to $\mathbf{n}_P$. Denote the equation of $H$ by
\[H:\mathbf{n}_H^T\cdot \mathbf{x}=c,\] where $\mathbf{n}_H$ is the normal vector and $\mathbf{x}\in\mathbb{R}^{n+r}$ . 

\begin{lemma}\label{lem: not vertical}
$\mathbf{n}_H^T\cdot\mathbf{n}_P\neq 0$.
\end{lemma}
\begin{proof}
Because $H$ contains $S'$, the equality
\begin{equation}\label{eq0}
\mathbf{n}_H^T\cdot (\mathbf{P}_i+h_i\cdot\mathbf{n}_P)=c
\end{equation}
holds for every vertex $\mathbf{P}_i$ of $S$. 

Assume by contradiction that $\mathbf{n}_H^T\cdot\mathbf{n}_P=0$. Then by (\ref{eq0}), we have $\mathbf{n}_H^T\cdot \mathbf{P}_i=c$
for every vertex $\mathbf{P}_i$ of $S$. This implies that all the vertices of $S$ are also on $H$, and hence $H$ coincides with the affine hyperplane spanned by $\mathcal{P}$. Thus $\mathbf{n}_H$ is parallel to $\mathbf{n}_P$, which contradicts $\mathbf{n}_H^T\cdot\mathbf{n}_P=0$. 

\end{proof}

Let \[\mathbf{w}=(h_{-1}-h_1, \cdots, h_{-n}-h_n)^T\in\mathbb{R}^n\] and\[\overrightarrow{B}=\chi(S).\]

\begin{lemma}\label{lem: lower face}
For any $\mathbf{P}_j$ that is not a vertex of $S$, the corresponding lifted point $\mathbf{P}_j'$ is higher than $H$ if and only if $\mathbf{w}^T\cdot\mathbf{C}_{-j}>0$, where $\mathbf{C}_{-j}$ is the signed fundamental circuit of the arc $\overrightarrow{e_{-j}}=\chi(\mathbf{P}_{-j})$ with respect to the basis $B$. 
\end{lemma}
\begin{proof}

We define
\[\overline{h}_j=\text{the unique real number satisfying }\mathbf{P}_j+\overline{h}_j\cdot\mathbf{n}_P\in H,\]where the existence of $\overline{h}_j$ is due to Lemma~\ref{lem: not vertical}.  
Intuitively, we use the new number $\overline{h}_j$ to lift $\mathbf{P}_j$ so that it lies in $H$. Therefore $h_j>\overline{h}_j$ if and only if $\mathbf{P}_j'$ is higher than $H$. 

Without loss of generality, we may assume $\mathbf{n}_H^T\cdot\mathbf{n}_P=1$ by rescaling the equation of $H$. Then the equation (\ref{eq0}) can be simplified to
\begin{equation}\label{eq1}
\mathbf{n}_H^T\cdot \mathbf{P}_i+h_i=c,
\end{equation}
which holds for every vertex $\mathbf{P}_i$ of $S$. By the definition of $\overline{h}_j$, we also have 
\begin{equation}\label{eq2}
\mathbf{n}_H^T\cdot \mathbf{P}_j+\overline{h}_j=c.
\end{equation}

Because $\mathbf{C}_{-j}$ is a signed circuit, we have
\begin{align}
& M\cdot \mathbf{C}_{-j}  = 0 \nonumber \\
\Rightarrow \quad & (\mathbf{P}_1-\mathbf{P}_{-1}, \cdots, \mathbf{P}_n-\mathbf{P}_{-n})\cdot \mathbf{C}_{-j} =  0  \nonumber \\
\Rightarrow \quad &   \mathbf{n}_H^T\cdot(\mathbf{P}_1-\mathbf{P}_{-1}, \cdots, \mathbf{P}_n-\mathbf{P}_{-n})\cdot \mathbf{C}_{-j} =  0 \nonumber \\
\Rightarrow \quad &   (\mathbf{n}_H^T\cdot\mathbf{P}_1-\mathbf{n}_H^T\cdot\mathbf{P}_{-1}, \cdots, \mathbf{n}_H^T\cdot\mathbf{P}_n-\mathbf{n}_H^T\cdot\mathbf{P}_{-n})\cdot \mathbf{C}_{-j} =  0. \label{eq3}
\end{align}
Here $(\mathbf{P}_1-\mathbf{P}_{-1}, \cdots, \mathbf{P}_n-\mathbf{P}_{-n})$ is an $(r+n)\times n$ matrix. 

In the last equality, we focus on the vertices $\mathbf{P}_i$ and $\mathbf{P}_{-i}$ such that the $i$-th entry of $\mathbf{C}_{-j}$ is non-zero since the other terms contribute zero to the left-hand side. These vertices correspond to the arcs in $\mathbf{C}_{-j}$ and $-\mathbf{C}_{-j}$. By the definitions of $\mathbf{C}_{-j}$ and $\overrightarrow{B}$, these vertices are all in $S$ except $\mathbf{P}_j$. Thus by (\ref{eq1}), (\ref{eq2}), and (\ref{eq3}), we have 
\[(h_{-1}-h_1,\cdots, h_{-j}-\overline{h}_j, \cdots, h_{-n}-h_n)\cdot \mathbf{C}_{-j} =  0, \text{ for }j>0;\]
\[(h_{-1}-h_1,\cdots,\overline{h}_{j}-h_{-j}, \cdots, h_{-n}-h_n)\cdot \mathbf{C}_{-j} =  0, \text{ for }j<0.\]

Hence we have \[\mathbf{w}^T\cdot\mathbf{C}_{-j}=(0,\cdots,\overline{h}_j-h_j,\cdots, 0)\cdot \mathbf{C}_{-j}, \text{ for }j>0;\]
 \[\mathbf{w}^T\cdot\mathbf{C}_{-j}=(0,\cdots,h_j-\overline{h}_j,\cdots, 0)\cdot \mathbf{C}_{-j}, \text{ for }j<0.\]

Note that the $|j|$-th entry of $\mathbf{C}_{-j}$ has the opposite sign to $j$. Therefore $h_j>\overline{h}_j$ if and only if $\mathbf{w}^T\cdot\mathbf{C}_{-j}>0$. 
\end{proof}

We are ready to prove the last main result of the paper. Recall that the map $\alpha:\sigma\mapsto\mathcal{A}_\sigma$ is a bijection between triangulating circuit signatures and triangulating external atlases of $\mathcal{M}$, and the map $\chi$ is a bijection between triangulations of $\mathcal{P}$ and triangulating external atlases.

\begin{theorem}[Theorem~\ref{main3}]
The restriction of the bijection $\chi^{-1}\circ\alpha$ to the set of acyclic circuit signatures of $\mathcal{M}$ is bijective onto the set of regular triangulations of $\mathcal{P}$. 
\end{theorem}

\begin{proof}
Recall that a circuit signature $\sigma$ is acyclic if and only if there exists $\mathbf{w}\in\mathbb{R}^n$ such that $\mathbf{w}^T \cdot \mathbf{C}>0$ for each signed circuit $\mathbf{C}\in\sigma$ (Lemma~\ref{geometric signature}).

Suppose we have a regular triangulation, which is induced by heights $h_i$, where $i=1, \cdots, n, -1, \cdots, -n$. We need to find an acyclic signature $\sigma$ such that the regular triangulation equals $\chi^{-1}({A}_\sigma)$. Let $\mathbf{w}$ be $(h_{-1}-h_1, \cdots, h_{-n}-h_n)^T$,  
and let $\sigma$ be the acyclic signature induced by $\mathbf{w}$ in the sense of Lemma~\ref{geometric signature}. Then it is enough to show that, for every maximal simplex $S\in\chi^{-1}({A}_\sigma)$, the corresponding lifted simplex $S'$ is a lower facet of the lifted polytope $\mathcal{P}'$. Still denote
$\overrightarrow{B}=\chi(S)$, where $\overrightarrow{B}$ is an externally oriented bases in $\mathcal{A}_\sigma$. Then for every vertex $\mathbf{P}_j$ that is not in $S$, the arc $\overrightarrow{e_{-j}}=\chi(\mathbf{P}_{-j})$ is in $\overrightarrow{B}$ because at least one of $\overrightarrow{e_{j}}$ and $\overrightarrow{e_{j}}$ is in $\overrightarrow{B}$. The signed fundamental circuit $\mathbf{C}_{-j}$ of $\overrightarrow{e_{-j}}$ with respect to $B$ is in $\sigma$, and hence we have $\mathbf{w}^T\cdot\mathbf{C}_{-j}>0$. By Lemma~\ref{lem: lower face}, the point $\mathbf{P}_j'$ is higher than the affine hyperplane $H$ spanned by $S'$. Because this holds for every vertex $\mathbf{P}_j$ that is not in $S$, the simplex $S'$ is a lower facet. 

Conversely, if we have an acyclic circuit signature induced by some vector $\mathbf{w}$, then we may construct the heights $h_i$ such that $\mathbf{w}=(h_1-h_{-1}, \cdots, h_n-h_{-n})^T$ and get a lifted polytope $\mathcal{P}'$. By a similar argument, one can show that for every maximal simplex $S\in\chi^{-1}({A}_\sigma)$, the corresponding lifted simplex $S'$ is a lower facet of $\mathcal{P}'$. Thus the triangulation $\chi^{-1}({A}_\sigma)$ is regular. 
\end{proof}
This completes proving the results in Section~\ref{Lawrence-intro}.

\section*{Acknowledgement}
Many thanks to Olivier Bernardi for orienting the author toward the study of this project and for countless helpful discussions. Thanks to Spencer Backman, Matt Baker, and Chi Ho Yuen for helpful discussions. Thanks to Matt Baker for helpful comments on the first draft of the paper. Thanks to Gleb Nenashev for proving Theorem~\ref{triangular}. Thanks to the referees for helpful suggestions.

\bibliographystyle{plain}
\bibliography{Lawrence.bib}

\end{document}